\documentclass[11pt]{amsart}

\usepackage{amsmath,amsfonts,amssymb,amsthm,enumerate,tikz-cd}

\usepackage{amsfonts,mathrsfs,amscd}

\usepackage{latexsym}
\usepackage{longtable}

\usepackage{shuffle}

\def\Z{\mathbb Z}
\def\Q{\mathbb Q}
\def\F{\mathbb F}

\newcommand{\PDiv}{\mathrm{PDiv}}

\newcommand{\End}{\mathrm{End}}
\newcommand{\Mat}{\mathrm{Mat}}
\newcommand{\Tr}{\mathrm{Tr}}
\newcommand{\AJ}{\mathrm{AJ}}
\newcommand{\Coker}{\mathrm{Coker}}

\newcommand{\Sym}{\mathrm{Sym}}
\newcommand{\Div}{\mathrm{Div}}
\newcommand{\R}{\mathbb{R}}
\newcommand{\MHS}{\mathrm{MHS}}
\newcommand{\Ext}{\mathrm{Ext}}
\newcommand{\divv}{\mathrm{div}}

\newcommand{\gr}{\mathrm{gr}}

\newcommand{\reg}{\mathrm{reg}}

\def\Aut{\mathrm{Aut}}

\def\cl{\mathrm{cl}}

\theoremstyle{plain}
\newtheorem{theorem}{Theorem}
\newtheorem{conjecture}{Conjecture}
\newtheorem{lemma}{Lemma}
\newtheorem{proposition}{Proposition}

\theoremstyle{definition}

\newtheorem{definition}{Definition}
\newtheorem{example}{Example}
\newtheorem{remark}{Remark}

\DeclareMathOperator{\ab}{ab}
\DeclareMathOperator{\un}{un}
\newcommand{\Jac}{\mathrm{Jac}}

\DeclareMathOperator{\Gal}{Gal}
\DeclareMathOperator{\dR}{dR}

\newcommand{\Pic}{\mathrm{Pic}}
\newcommand{\colim}{\mathrm{colim}}
\DeclareMathOperator{\loc}{loc}

\newcommand{\NS}{\mathrm{NS}}
\newcommand{\Ker}{\mathrm{Ker}}
\newcommand{\Hom}{\mathrm{Hom}}

\newcommand{\CH}{\mathrm{CH}}

\newcommand{\et}{\mathrm{et}}
\newcommand{\Alb}{\mathrm{Alb}}

\newcommand{\cris}{\mathrm{cris}}
\DeclareMathOperator{\ext}{Ext}

\usepackage{mathrsfs}
\usepackage[OT2,T1]{fontenc}
\DeclareSymbolFont{cyrletters}{OT2}{wncyr}{m}{n}
\DeclareMathSymbol{\Sha}{\mathalpha}{cyrletters}{"58}

\begin{document}
\title{Generalised height pairings and the Albanese kernel}
\author{Netan Dogra}

\begin{abstract}
The Chabauty--Coleman--Kim method in depth two describes the rational points on a curve in terms of a generalisation of Nekov\'a\v r's $p$-adic height pairing which replaces $\mathbb{G}_m$ with a higher Chow group. It is unclear both what the domain of definition of this pairing is, and how to compute it. This paper explores the relevance of the Beilinson--Bloch conjectures to this problem. In particular, it is shown that if $X$ is a smooth projective curve and the Albanese kernel of $X\times X$ is torsion, then there is an algorithm to compute the generalised height pairing on a pair of rational points on the Jacobian. This leads to the consideration of certain `motivic refinements' of the nonabelian cohomology varieties which arise in nonabelian Chabauty.
\end{abstract}
\dedicatory{To Minhyong Kim on the occasion of his 61st birthday}

\maketitle

\tableofcontents
\section{Introduction}
The Bloch--Kato conjectures \cite{BK} make a striking prediction about the space of extensions of Galois representations which come from geometry. More precisely, if $V$ and $W$ are finite dimensional Galois representations we can consider the set of isomorphism classes of short exact sequences of Galois representations
\[
0\to W\to E\to V\to 0.
\]
This forms a group under Baer summation, isomorphic to the Galois cohomology group $H^1 (K,V^* \otimes W)$. In the case where $V=\Q _p $ and $W=H^{2i-1}_{\et }(X_{\overline{\Q }},\Q _p (i))$, where $X$ is a smooth projective variety over $\Q $, an extension comes from geometry if it is in the image of the map
\[
\cl _p :\CH ^i (X)_0 \otimes \Q _p \to \Ext ^1 _{\Gal (K)}(\Q _p ,H^{2i-1}(X_{\overline{K}},\Q _p (i)))
\]
which is given (\cite{deligne:valeurs}, \cite[II.9]{jannsen1990}) by sending a homologically trivial cycle $Z=\sum n_i Z_i $ to the extension of $\Q _p $ by $H^{2i-1}_{\et }(X_{\overline{\Q }},\Q _p (i))$ obtained by pulling back the exact sequence
\[
0\to H^{2i-1}_{\et }(X_{\overline{\Q }},\Q _p (i))\to H^{2i-1}_{\et }(X_{\overline{\Q }}-|Z| ,\Q _p (i))\to H^{2i} _{| Z| }(X_{\overline{\Q }},\Q _p (i))
\]
along the cycle class $\Q _p \to H^{2i}_{|Z|}(X_{\overline{\Q }},\Q _p (i))$ of $Z$ (where $|Z|:=\cup _i Z_i $). Bloch and Kato conjecture that every extension of $\Q _p $ by $H^{2i-1}_{\et }(X_{\overline{\Q }},\Q _p (i))$ unramified outside a finite set of primes, and de Rham at $p$, can be written as a $\Q _p$-linear combination of elements of $\CH ^i (X)_0 $, and furthermore that such a presentation is unique modulo torsion (i.e. the map $\cl _p $ is injective).

In Kim's nonabelian generalisation of the Chabauty--Coleman method, one is instead interested in \textit{iterated} extensions (henceforth referred to as \textit{mixed} extensions) of geometric Galois representations. The graded pieces of these iterated extensions are quotients of tensor powers of the Tate module of the Albanese of the original variety $X$. The utility of working with this space of iterated extensions is that, while the space of geometric extensions of $\Q _p $ by $H^1 _{\et }(X_{\overline{\Q }},\Q _p (1))$ may be large, conjectures on the dimensions of global Galois cohomology groups predict the space of geometric extensions of $\Q _p $ by higher tensor powers are `small', and in particular most local iterated extensions do not come from global iterated extensions, giving an obstruction to a $p$-adic point coming from a rational point.

More precisely, let $b\in X(\Q )$ and let $U_n $ be the maximal $n$-unipotent quotient of the $\Q _p $-unipotent completion of the \'etale fundamental group of $X_{\overline{\Q }}$ with basepoint $b$. Then one defines $X(\Q _p )_n \subset X(\Q _p )$ to be the set of points extending to an adelic point $(x_v )\in X(\mathbb{A})$ whose class in $\prod _v H^1 (\Q _v ,U_n )$ lies in the (scheme-theoretic) image of $H^1 (\Q ,U_n )$. More generally, for any Galois stable quotient $U$ of $U_n$, one can similarly define a subset $X(\Q _p )_U $ of $X(\Q _p )$.

\subsection{Applications to quadratic Chabauty}
In this article we will only discuss the depth two case of nonabelian Chabauty. This is sometimes referred to as quadratic Chabauty, although it may be helpful to first disambiguate that term slightly. Let $V:=\Q _p \otimes _{\Z _p }T_p \Jac (X)$. Let $U_2 ^T$ be the maximal quotient of $U_2 $ whose commutator subgroup is isomorphic to a direct sum of copies of $\Q _p (1)$, and let $U_2 ^A$ be the maximal quotient of $U_2 $ whose commutator subgroup is a twist of an Artin representation by $\Q _p (1)$. Then we have a commutative diagram with surjective vertical maps and exact rows
\[
\begin{tikzcd}
1 \arrow[r] & \Ker (H^2 (J_{\overline{\Q }},\Q _p )\stackrel{\AJ ^* }{\longrightarrow }H^2 (X_{\overline{\Q }},\Q _p ))^* \arrow[d] \arrow[r] & U_2 \arrow[d] \arrow[r] & V \arrow[d] \arrow[r] & 1 \\
1 \arrow[r] & \Ker (\NS (J_{\overline{\Q }})\stackrel{\AJ ^* }{\longrightarrow }\NS (X))^* \otimes \Q _p (1)\arrow[d] \arrow[r] & U_2 ^A \arrow[d] \arrow[r] & V \arrow[d] \arrow[r] & 1 \\
1 \arrow[r] & \Ker (\NS (J_{\Q })\stackrel{\AJ ^* }{\longrightarrow }\NS (X))^* \otimes \Q _p (1) \arrow[r]           & U_2 ^T \arrow[r]           & V \arrow[r]           & 1.
\end{tikzcd}
\]
Where necessary, to ease notation we will refer to the sets $X(\Q _p )_{U_2 ^T},X(\Q _p )_{U_2 ^A}$ and $X(\Q _p )_2 $ as the skimmed, semi-skimmed and full-fat quadratic Chabauty loci.

In the most well-studied case of skimmed quadratic Chabauty, equations for $X(\Q _p )_{U_2 ^T}$ are obtained in terms of the cyclotomic $p$-adic height \cite{nekovar}. Recall that in Nekov\'a\v r's approach to $p$-adic heights \cite{nekovar}, a local $p$-adic height is a function on local mixed extensions (i.e. iterated extensions in the category of Galois representations of a local field) with graded pieces $\Q _p ,V$ and $\Q _p (1)$. Summing over all places gives a global height, which has a bilinearity property. From the perspective of nonabelian Chabauty, the bilinearity of the global height is a reciprocity law giving a condition for a collection of local Galois representations to come from a global one, and hence for an adelic point to come from a rational point.

In semi-skimmed quadratic Chabauty one does the same thing with $\Q _p (1)$ replaced by $\Q _p (1)\otimes \rho _0 $ where $\rho _0$ is an Artin representation. The reciprocity law then involves $p$-adic heights over number fields associated to more general idele class characters. 

In full-fat quadratic Chabauty, one assigns to each pair of rational points on a smooth projective variety $X$ a Galois representation which is filtered with graded pieces $\Q _p ,V:=H^1 _{\et }(X_{\overline{\Q}},\Q _p (1))$ and $\Coker (H^1 _{\et }(X_{\overline{\Q}},\Q _p )^{\otimes 2}\stackrel{\cup }{\longrightarrow }H^2 _{\et }(X_{\overline{\Q }},\Q _p ))^* $. We write $M_g (\Q ;\Q _p ,V,V^{\otimes 2}/\Q _p (1))$ for the set of such iterated extensions, with suitable local conditions (see \S \ref{sec:me} for details). This is a nonabelian cohomology variety which is slightly larger than the Selmer scheme considered by Kim. However, working with this larger space has benefits: it behaves more like the Selmer scheme of a elliptic curve minus the origin with tangential basepoint studied by Kim in \cite{kim:massey}. The reciprocity law giving a condition for a collection of local mixed extensions to come from a global mixed extension is then the \textit{generalised height} pairing introduced in \cite{QC2}. This has been used in \cite{QC2} to determine rational points on bihyperelliptic curves of the form $y^2 =x^6+ax^4+ax^2 +1$, and in \cite{dogra2023} to determine the rational points on the hyperelliptic curve $y^2-y=x^5-x$. 

There are many obstacles to turning full-fat quadratic Chabauty into a practical method for computing these obstructions. The first is verifying these conjectures on dimensions of Galois cohomology groups, i.e. verifying the expectation that most local representations do not come from global ones.\footnote{see \cite{berry2025refined} for recent progress.} The second is computing the generalised height pairing.

To explain why this is a difficult problem, consider the cases of classical Chabauty--Coleman and skimmed quadratic Chabauty. Let $X$ denote a smooth projective geometrically irreducible curve, and $J$ its Jacobian. In the case of classical Chabauty, one determines the Chabauty--Coleman locus $X(\Q _p )_1$ by finding a set of zero-cycles on $X$ generating a finite index subgroup of $J(\Q )$. In skimmed quadratic Chabauty, one needs to determine the values of the $p$-adic height pairing. By analogy with the abelian case, one might hope that one can do this by evaluating the $p$-adic height function on points over number fields. However it is far from obvious that this works: for example in the case of integral points on a rank $=$ genus hyperelliptic curve, one can determine the $p$-adic height pairing from $K$-points on $X$ (for $K/\Q $ a Galois extension, possibly infinite) if and only if $\Sym ^2 J(\Q )\otimes \Q $ is contained in the subgroup of $(\Sym ^2 J(K)\otimes \Q )^{\Gal (K|\Q )}$ generated by $(x-\infty )^2 $, for $x\in X(K)$ and $\infty $ a Weierstrass point. 

In practice one instead computes the $p$-adic height pairing using a more general construction -- e.g. using algorithms for computing local height pairings between divisors $D_1 $ and $D_2$ with disjoint support. From the perspective of mixed extensions, this can be interpreted as constructing all mixed extensions with graded pieces $\Q _p ,V,\Q _p (1)$ coming from geometry via subquotients of Galois representations of the form $H^1 _{\et }(X_{\overline{\Q }}-|D_1 |;|D_2 |)(1)$ (see \cite{nekovar}, \cite{scholl}).

\subsection{The Albanese kernel}
In this paper we explain how to compute the generalised height pairings arising in quadratic Chabauty assuming the conjecture of Beilinson--Bloch that the Albanese kernel of $X\times X$ is torsion. Recall that the Albanese kernel $T(X^2 )$ is the kernel of the map
\[
\CH ^2 (X^2 )_0 \to \Jac (X)^2 
\]
defined by the two projections $X^2 \to X$.

Let $h$ denote the generalised height on mixed extensions with graded pieces $\Q _p ,V$ and $W:=\wedge ^2 V/(\NS (J_{\overline{\Q }})^* \otimes \Q _p (1))$ (see section \ref{sec:ht}). This may be viewed either as a function on $X(\Q )$, or as a function on $H^1 _f (\Q ,V)\times H^1 _f (\Q ,V^* \otimes W)$, or as a function on $J(\Q )\times \CH ^2 (X^3 )_0$. As explained below, the relation between these is via the cycle class map
\[
J(\Q )\times \CH ^2 (X^3 )_0 \to H^1 (\Q ,V)\times  H^1 (\Q ,V^* \otimes W)
\]
and the map
\[
X(\Q )\to J(\Q ) \times \CH ^2 (X^3 )_0
\]
sending $z$ to $(z-b,\iota _1 (z-b)+\Delta (b))$. Here $\Delta (b)\in \CH ^2 (X^3 )_0$ is the Gross--Schoen cycle associated to $b\in X(\Q )$ and $\iota _1 $ is a map $J(\Q ) \to \CH ^2 (X^3 )_0$. In \S \ref{subsec:curve} a definition is given of a generalised height $h_{\iota }(D)$ on $X$ associated to a principal divisor $D$ on a curve $\iota :C \hookrightarrow X\times X$.  To ease notation, we will suppress the $\iota _1 $, and write the generalised height pairing of $P$ in $J(\Q )$ with $\iota _1 (Q) +Z$, for $Q\in J(\Q )$ and $Z\in \CH^2 (X^3 )_0 $, if it is defined, as
\[
h(P,Q+Z).
\]
If $Z=0$ and $P=Q$, we will write this pairing as $h(P)$.
\begin{theorem}\label{thm:only_one}
Assuming the Beilinson--Bloch conjecture, there is an algorithm which, given a pair of points $P_1 , P_2 \in J(\Q ) $, returns integers $m_i ,n_j $, points $Q_j \in X(\overline{\Q })$, a finite set of curves $\iota _i :C_i \hookrightarrow X_{\overline{\Q }}\times X_{\overline{\Q }}$, and principal divisors $D_i $ on $C_i$, such that the generalised height pairing $h(P_1 ,P_2+\Delta (b))$ is equal to
\[
\sum \frac{1}{m_i }h_{\iota _i }(D_i )+\sum n_j h(Q_j ).
\]
\end{theorem}
%
%

One of the original visions of the Chabauty--Coleman--Kim method (or ``recurring fantasies'', in the language of the Beilinson--Bloch conjectures) was that the Bloch--Kato conjectures should imply the existence of an algorithm for determination of the rational points on curves \cite{kim:remarks}. As is pointed out in loc. cit., by analogy with the Birch--Swinnerton-Dyer conjecture one would actually hope for an algorithm whose \textit{termination} is conditional on these conjectures, meaning that one never needs to verify the Beilinson--Bloch conjectures for a variety in order to run the algorithm, but if the conjectures are false then the process is not guaranteed to terminate.

The present paper can be seen as an elaboration on one special case of that idea.\footnote{With the caveat that Chabauty methods are essentially \textit{never} known to give algorithms for determining rational points, because of well documented problems with possible nonrational points in $X(\Q _p )_n $ or multiplicities of zeroes of the relevant $p$-adic power series (see \cite{BS10} for a comprehensive description).} Namely, we show that the problem of determining equations for $X(\Q _p )_2 $ essentially reduces to verifying (or more accurately, exhibiting) that certain elements of the Albanese kernel of $X\times X$ are torsion. To explain this, 

\subsection{Motivic avatars}
Theorem \ref{thm:only_one} is proved by showing how to explicitly construct mixed extensions coming from geometry from fundamental groups of curves inside $X\times X$, or equivalently (by Beilinson) from Galois representations of the form $H^2 _{\et }(X^2 _{\overline{\Q }}-|C|,\Q _p (2))$, where $C$ is a cycle of codimension one on $X^2$. This suggests the possibility of constructing a `motivic avatar' of the Selmer scheme, or more general nonabelian cohomology varieties arising in the Chabauty--Coleman--Kim method, which can be used to compute the localisation map (and hence the equations for $X(\Q _p )_n$). In section \ref{sec:alb}, we explain how the proof of Theorem \ref{thm:only_one} can be viewed in terms of the construction of a set $\mathcal{M}(X\times X)$ sitting in a short exact sequence
\[
\CH ^2 (X^2 ,1)\to \mathcal{M}(X\times X)\to \Jac (X)^2 \stackrel{\boxtimes }{\longrightarrow } T(X^2 )
\]
which is compatible with the usual 7-term exact sequence for a central extension in nonabelian group cohomology (see Proposition \ref{prop:bigdiagram}).

In \cite{EL23}, Edixhoven and Lido observe that the motivic avatar of the set of mixed extensions arising in the theory of $p$-adic heights is exactly the Poincar\'e torsor $P^\times $, which is a $\mathbb{G}_m $-torsor over $J\times J$
\begin{equation}\label{eqn:poincare}
1\to \mathbb{G}_m \to P^\times \to J\times J\to 1.
\end{equation}
This gives a `motivic' description of skimmed quadratic Chabauty. In practice it is often easier to compute the $p$-adic height pairing by working direct with pairs of divisors with disjoint support rather than a model of the Poincar\'e torsor, in the same way that often in classical Chabauty--Coleman one works with divisors on the curve rather than a model of the Jacobian. 

One can extend the Poincar\'e torsor to describe semi-skimmed quadratic Chabauty. One can think of this as involving working with $T$-torsors $\widetilde{P}^\times $ on $J$, or other abelian varieties, where $T$ is a nonsplit torus, or alternatively as simply working with the Poincar\'e torsor over number fields. However, beyond semi-skimmed quadratic Chabauty one would not expect any kind of motivic avatar of the Selmer scheme. A precise statement in this regard is the recent paper of Rogov which proves that when $n\geq 3$, Hain's depth $n$ higher Albanese manifold (for a hyperbolic curve $X/\mathbb{C}$) is not algebraisable \cite{rogov2025minimal}. An archetypal example of the kind of non-algebraic structure that can arise is that of `mixed Tate quadratic Chabauty' -- i.e. the study of $X(\Z _p )_2$ when $X$ is $\mathbb{P}^1 -\{ 0,1,\infty \}$. Here the relevant bilinear form comes from the $p$-adic dilogarithm function \cite{DCW}. Then there is a precise analogue of the exact sequence \eqref{eqn:poincare} defined in terms of the abelian category of mixed Tate motives. However, from a computational perspective, arguably a more useful analogue is the related (but slightly different) Bloch--Suslin complex.\footnote{perhaps a strict analogue would be a short exact sequence of pointed sets constructed from the Bloch--Suslin complex, see \S \ref{subsec:BS}.} Instead of being short exact, this sequence has four terms, with the Milnor $K$-group $K^M _2 (\Q )$ providing an obstruction analogous to the obstruction coming from the Albanese kernel. In summary we have the following table of (sometimes imprecise) analogies.

\begin{center}
\begin{tabular}{ c| c c c}
  & Galois representations & motivic avatar & obstruction \\ \hline
 Chabauty & $\left( \begin{array}{cc}1 & 0 \\ * & \rho _V \end{array} \right) $ &  $\Jac (X)$ & \shortstack{$\Sha (J)$} \\  
\shortstack{skimmed \\ quadratic \\Chabauty} & $\left( \begin{array}{ccc}1 & 0 & 0 \\ * & \rho _V & 0  \\ * & * & \chi \end{array} \right) $ & $P^\times $     &   \\ 
\shortstack{semi-skimmed \\ quadratic \\Chabauty} & $\left( \begin{array}{ccc}1 & 0 & 0 \\ * & \rho _V & 0  \\ * & * & \rho _0 \otimes \chi \end{array} \right) $ & \shortstack{twisted \\ Poincar\'e torsor}     &  $\Sha (T)$  \\ 
\shortstack{mixed Tate \\ quadratic \\Chabauty} & $\left( \begin{array}{ccc}1 & 0 & 0 \\ * & \chi & 0  \\ * & * & \chi ^2 \end{array} \right) $ & \shortstack{Bloch--Suslin \\ complex} & $K^M _2 (\Q )$ \\  
\shortstack{full-fat \\ quadratic \\Chabauty} & $\left( \begin{array}{ccc}1 & 0 & 0 \\ * & \rho _V & 0  \\ * & * & \rho _V ^{\otimes 2} \end{array} \right) $ & $\mathcal{M}(X^2 )$ & $T(X^2)$ \\  
\end{tabular}
\end{center}
Note that the notion of `motivic avatars' we are discussing is \textit{not}, a priori, the same as the notion of `motivic' nonabelian Chabauty--Kim which arises in work of Dan--Cohen, Wewers and Corwin \cite{DCW} \cite{CDC}. Recall that if there were a Tannakian category of mixed motives over $\Q $, then one could view the space of mixed extensions which come from motives as nonabelian cohomology of a motivic Galois group (the fundamental group of the Tannakian category), and to try to compute the image of the map from nonabelian cohomology of the motivic Galois group to nonabelian cohomology of the Galois group of $\Q _p $. There are two obstacles to this approach.
\begin{enumerate}
\item We do not know if there is a Tannakian category of mixed motives.
\item Even if we did, it is not clear how to compute the $p$-adic periods of its objects without some explicit geometric realisation.
\end{enumerate}
To elaborate briefly on the second point: Dan--Cohen, Wewers, Corwin \cite{DCW} \cite{CDC} and others have used the Tannakian category of mixed Tate motives to great effect in the study of Chabauty--Kim for the $S$-unit equation. Whilst the existence of such a category greatly clarifies structural properties, as far as the author is aware one still needs a geometric realisation of these mixed Tate motives in order to compute their $p$-adic periods. However, perhaps using an explicit description of the Galois group of mixed Tate motives in terms of algebraic cycles would give a practical approach to this computational problem \cite{bloch}.

\subsection{Conventions}
Group cohomology of a profinite group refers to continuous group cohomology. If $K$ is a finite extension of $\Q _{\ell }$ with residue field $k$ and $W$ is a finite dimensional $\Q _p $-representation of $G_K$, we write $H^1 _f (K ,W)$ to mean $H^1 (G_k ,W^{I_K})$ if $\ell \neq p$ and $\Ker (H^1 (G_K ,W)\to H^1 (G_K ,W\otimes B_{\cris }))$ if $\ell = p$. We write $H^1 _g (K ,W)$ to mean $H^1 (K,W)$ if $\ell \neq p$ and $\Ker (H^1 (K,W)\to H^1 (K,W\otimes B_{\dR}))$ if $\ell = p $. If $K$ is a number field and $W$ is a $\Q _p$-representation unramified outside all but a finite set $S_0$ of primes and de Rham at all primes dividing $p$, and $S$ is a finite set of primes, we write $H^1 _{f,S}(K,W)$ to mean the subspace of $H^1 (G_{K,S\cup S_0 },W)$ consisting of classes whose localisation at a prime $v$ lies in $H^1 _f (G_{K_v},W)$ if $v$ is not in $S$, and lies in $H^1 _g (G_{K_v },W)$ if $v$ lies in $S$. If $S$ is empty, this is written simply as $H^1 _f (K,W)$. Finally, we write  $H^1 _g (K,W)$ to mean the direct limit, over finite subsets $S$ of primes of $K$, of $H^1 _{f,S}(K,W)$.

To lighten notation, if $X$ is a variety over $K$, in this article we will always write $H^i _{\et }(X,\Q _p (j))$ to mean $H^i _{\et }(X_{\overline{K}},\Q _p (j))$. Apologies for any confusion caused.
\subsection*{Acknowledgements}
It is a pleasure to be able to thank Minhyong Kim for the opportunity to spend so much time in the mathematical world he has created, and for his generosity and optimism. This paper benefited from many discussions with Carl Wang-Erickson about mixed extensions, and with Jonathan Love about the Albanese kernel. I am extremely grateful to the anonymous referee for numerous corrections to an earlier version of this paper.

This research is supported by a Royal Society university research fellowship.
\section{The generalised height}\label{sec:ht}

\subsection{Mixed extensions}\label{sec:me}
Given objects $V_0 ,\ldots ,V_n$ in an abelian category $\mathcal{C}$, we can define a mixed extension with graded pieces $V_0 ,\ldots ,V_n$ to be an object $W$ of $\mathcal{C}$ together with a filtration by subobjects
\[
W=W_0 \hookleftarrow W_1 \hookleftarrow \ldots \hookleftarrow W_{n+1} =0
\]
together with isomorphisms $\psi _i :V_i \simeq W_{i}/W_{i+1}$. We define a morphism of mixed extensions $(W,(W_i ),(\psi _i ))\to (W',(W'_i ),(\psi ' _i ))$ to be an isomorphism $W\to W'$ in $\mathcal{C}$, respecting the filtrations and commuting with the $\psi _i$ and $\psi _i '$. We write $\mathcal{M}(\mathcal{C};V_0 ,\ldots ,V_n )$ to be the category of mixed extensions with graded pieces $V_0 ,\ldots ,V_n $ in $\mathcal{C}$, and write $M(\mathcal{C};V_0 ,\ldots ,V_n )$ for $\pi _0 $ of this category (i.e. the set of isomorphism classes of mixed extensions with these graded pieces).

If $\mathcal{C}$ is the category of continuous finite dimensional $\Q _p$-representations of a group $G$, then we can define a group $U(V_0 ,\ldots ,V_n )$ of unipotent automorphisms\footnote{here we view $\oplus _{i=0}^n V_i$ as a mixed extension of the $V_i$ in the obvious way. A unipotent automorphism of a mixed extension is an automorphism respecting the filtration and acting as the identity on the associated graded (i.e. an endomorphism of mixed extensions).} of $V_0 \oplus \ldots \oplus V_n $ as a $\Q _p $-vector space, and we have a bijection
\begin{equation}\label{eqn:iso_coh}
M(\mathcal{C};V_0 ,\ldots ,V_n )\simeq H^1 (G,U(V_0 ,\ldots ,V_n )).
\end{equation}
For example, when $n=1$, we have $M(\mathcal{C};V_0 ,V_1 )\simeq \Ext ^1 _G (V_0 ,V_1 )$, and we recover the usual isomorphism
\[
\Ext ^1 _G (V_0 ,V_1 )\simeq H^1 (G,V_0 ^* \otimes V_1 ).
\]
In the case $n=2$, we have an exact sequence
\begin{equation}\label{eqn:ES}
H^1 (G,V_0 ^* \otimes V_2 )\to H^1 (G,U(V_0 ,V_1 ,V_2 ))\to H^1 (G,V_0 ^* \otimes V_1 )\times H^1 (G,V_1 ^* \otimes V_2 )\to H^2 (G,V_0 ^* \otimes V_2 ).
\end{equation}
\begin{lemma}
The boundary map
\[
H^1 (G,V_0 ^* \otimes V_1 )\times  H^1 (G,V_1 ^* \otimes V_2 )\to H^2 (G,V_0 ^* \otimes V_2 )
\]
associated to the central extension
\[
1\to V_0 ^* \otimes V_2 \to U(V_0 ,V_1 ,V_2 )\to V_0 ^* \otimes V_1 \oplus V_1 ^* \otimes V_2 \to 1
\]
is given by the usual cup product. 
\end{lemma}
\begin{proof}
This can be checked on the level of cochains, see e.g. \cite[Lemma 20]{dogra2023}.
\end{proof}

By composing the $\psi _i$ with a scalar, we get an action of $(\Q _p ^\times )^n $ on $M(\mathcal{C};V_0 ,\ldots ,V_n )$. More generally, we have an action of $\prod _{i=1}^n \Aut _G (V_i )$ on $M(\mathcal{C};V_0 ,\ldots ,V_n )$. Under \eqref{eqn:iso_coh}, this corresponds to the action of $\prod _{i=1}^n \Aut _G (V_i )$ on $H^1 (G,U(V_0 ,\ldots ,V_n ))$ via the obvious map
\[
\prod _{i=1}^n \Aut _G (V_i )\to \Aut _G (U(V_0 ,\ldots ,V_n )).
\]

\subsection{Biextension structure}\label{biextensionstructure}
\begin{definition}\label{defn:mixex}
If $E_1 $ and $E_2 $ are extensions of $V_0$ by $V_1$ and $V_1$ by $V_2$ respectively, we will say that a mixed extension $M$ with graded pieces $V_0 ,V_1 ,V_2 $ such that $\pi _1 (M) =[E_1 ]$ and $\pi _2 (M)= [E_2 ]$ is a mixed extension of $[E_1 ]$ and $[E_2 ]$ (where $\pi _i $ are the obvious projection maps).
\end{definition}
If we fix $E_1$, then the set of all classes of mixed extensions in $M(V_0 ,V_1 ,V_2 )$ with $\pi _1 (M)= [E_1 ]$ is in bijection with the group $\Ext ^1 (E_1 ,V_2 )$. Similarly, if we fix $E_2 $ then the set of all mixed extensions with $\pi _2 (M)= [E_2 ]$ is in bijection with the group $\Ext ^1 (V_0 ,E_2 )$. If $M_1 $ and $M_2 $ are two mixed extensions with $\pi _i (M_1 )=\pi _i (M_2 )$ ($i=1$ or $2$) then we will denote by $M_1 +_i M_2 $ the mixed extension obtained by Baer summation in this ext group.

We furthermore have an action of $H^1 (G,V_0 ^* \otimes V_2 )$ on $M(G;V_0 ,V_1 ,V_2 )$. On the level of representations, for a cocycle $c \in Z^1 (G,V_0 ^* \otimes V_2 )$, and a mixed extension with underlying Galois representation
\[
\rho =\left( \begin{array}{ccc} \rho _{V_0 } & 0 & 0 \\ \kappa _1 & \rho _{V_1 }& 0 \\ \kappa _3 & \kappa _2 & \rho _{V_2 } \end{array} \right) ,
\]
this is given by sending $\rho $ to the representation
\[
c\cdot \rho :=\left( \begin{array}{ccc} \rho _{V_0 } & 0 & 0 \\ \kappa _1 & \rho _{V_1 }& 0 \\ \kappa _3 +c \cdot \rho _{V_0 } & \kappa _2 & \rho _{V_2 } \end{array} \right) ,
\]

Note that $M(G;V_0 ,V_1 ,V_2 )$ does not in general satisfy the axioms of a bi-extension of abelian groups in the sense of \cite{MumBi} or \cite{SGA7}, since we do not assume that the map
\[
M(V_0 ,V_1 ,V_2 )/\Ext ^1 (V_0 ,V_2 )\to \Ext ^1 (V_0 ,V_1 )\times \Ext ^1 (V_1 ,V_2 )
\]
is surjective. Indeed the possible failure of surjectivity is one of the main topics of this paper.

\subsection{Examples}
The most important examples for us will be when $\mathcal{C}$ is the category of geometric Galois representations over a local or global field of characteristic zero. That is, we are one of the following four situations:
\begin{enumerate}
\item $\mathcal{C}$ is the category of finite dimensional continuous $\Q _p$-representations of the Galois group of a finite extension of $\Q _{\ell }$, $\ell \neq p$. We will denote the corresponding categories by $\mathcal{M}_g (V_0 ,V_1 ,V_2 )$, or sometimes $\mathcal{M}_{g,v}(V_0 ,V_1 ,V_2 )$, and the sets of isomorphism classes by $M_g (V_0 ,V_1 ,V_2 )$, or sometimes $M_{g,v}(V_0 ,V_1 ,V_2 )$
\item $\mathcal{C}$ is the category of finite dimensional de Rham, or crystalline, representation of the Galois group of a finite extension of $\Q _p $. We will denote these categories by $\mathcal{M}_g (V_0 ,V_1 ,V_2 )$ and $\mathcal{M}_f (V_ 0,V_1 ,V_2 )$ respectively, and the set of isomorphism classes by $M_g (V_0 ,V_1 ,V_2 )$ and $M_f (V_ 0,V_1 ,V_2 )$ respectively
\item $\mathcal{C}$ is the category of finite dimensional $\Q _p $-representations of $G_{K,S}$ (the maximal quotient of the Galois group of a number field $K$ unramified outside a finite set of primes $S$) which are de Rham at all primes above $p$. We will denote the categories by $\mathcal{M}_g (G_{K,S} ;V_0 ,V_1 ,V_2 )$ and the isomorphism classes by $M_g (G_{K,S} ;V_0 ,V_1 ,V_2 )$.
\item $\mathcal{C}$ is the category of filtered $\phi $-modules over a finite extension of $\Q _p $. We will denote this category by $\mathcal{M}(V_0 ,V_1 ,V_2 )$ and the set of isomorphism classes by $M(V_0 ,V_1 ,V_2 )$.
\end{enumerate}
Note that in case 2, if the $V_i$ are crystalline (resp. de Rham), then $M_f (V_0 ,\ldots ,V_n )$ (resp. $M_g (V_0 ,\ldots ,V_n )$) may be identified with $H^1 _f (K,U(V_0 ,\ldots ,V_n ))$ (resp. $H^1 _g (K,U(V_0 ,\ldots ,V_n ))$), where $K$ is the underlying finite extension of $\Q _p$, and the cohomology sets $H^1 _* (K,U(V_0 ,\ldots ,V_n ))$ are in the sense of Kim \cite{kim:siegel}. The identification comes from restricting the isomorphism \eqref{eqn:iso_coh}.

Now let $X$ again be a smooth projective curve over a number field $K$, $J:=\Jac (X)$, $V:=T_p J\otimes _{\Z _p }\Q _p $ and $W$ a $\Gal (K)$-stable direct summand of $V^{\otimes 2}$.
\begin{lemma}\label{lemma:WM}
If $v|p$, then
\[
H^1 _f (K_v ,V)=H^1 _g (K_v ,V)
\]
and 
\[
H^1 _f (K_v ,V^* \otimes W)=H^1 _g (K_v ,V^* \otimes W).
\]
\end{lemma}
\begin{proof}
Recall that $H^1 _g (K_v ,V)/H^1 _f (K_v ,V)\simeq D_{\cris }(V) ^{\phi =1}$, and similarly for $V^* \otimes W$. Hence both statements are consequences of the weight monodromy conjecture for abelian varieties \cite{nekovar}.
\end{proof}
\begin{lemma}\label{lemma:zero}
If $v$ is prime to $p$, then $H^i (K_v ,V)$ and $H^i (K_v ,V^* \otimes W)=0$ are zero for all $i$, and the map
\[
H^1 _g (\Q _v ,W)\to M_{g,v}(\Q _p ,V,W)
\]
is an isomorphism.
\end{lemma}
\begin{proof}
This is also a consequence of the weight monodromy conjecture for abelian varieties (see \cite{nekovar} or \cite[Lemma 3.4]{QC2}).
\end{proof}

\begin{lemma}\label{lemma:SES}
In cases (1) and (2) above, the exact sequence \eqref{eqn:ES} extends to a short exact sequence of pointed sets
\[
1\to H^1 _g (K _v ,W)\to M_g (\Q _p ,V,W)\to H^1 _g (K_v ,V)\times H^1 _g (K_v ,V^* \otimes W) \to 1.
\]
\end{lemma}
\begin{proof}
In the case where $v$ is prime to $p$ this follows from Lemma \ref{lemma:zero}.

In the case where $v$ divides $p$, this is proved in \cite{QC2} in the case of crystalline representations. We recall the argument: one shows that there is a commutative diagram with exact rows whose vertical maps are all isomorphisms
\[
\begin{small}
\begin{tikzcd}
\ext ^1 _f (\Q _p ,W) \arrow[r] \arrow[d] & M_f (\Q _p ,V,W) \arrow[r] \arrow[d] & \ext ^1 _f (\Q _p ,V)\times \ext ^1 _f (V, W) \arrow[d] \\
\ext ^1  (\Q _p ,D(W)) \arrow[r] \arrow[d] &  M (\Q _p ,D(V),D(W)) \arrow[r] \arrow[d] & \ext ^1  (\Q _p ,D(V))\times \ext ^1 (D(V), D(W)) \arrow[d] \\
D_{\dR}(W)/F^0 \arrow[r]           & U(K_v ,D_{\dR}(V),D_{\dR}(W))/F^0 \arrow[r]           & D_{\dR}(V)/F^0 \times D_{\dR}(V^* \otimes W)/F^0          
\end{tikzcd}
\end{small}
\]
where $D(.):=D_{\cris }(.)$, the ext groups and mixed extensions in the second row are in the category of filtered $\phi $-modules over $K_v$, the middle vertical maps are given by the $D_{\cris }$ functors, and the bottom vertical maps are given by Kim's classification of unipotent torsors in the category of filtered $\phi $-modules (see \cite[\S 3.3]{QC2}). Since the bottom extends to a short exact sequence, so does the top.

If we replace crystalline by de Rham we still get surjectivity of $M_g (\Q _p ,V,W) \to H^1 _f (K_v ,V)\times H^1 _f (K_v ,V^* \otimes W) $, since both $V$ and $V^* \otimes W$ have the property that $H^1 _f =H^1 _g $ by Lemma \ref{lemma:WM}. On the other hand injectivity of $H^1 _g (K_v ,W)\to H^1 (K_v ,U(\Q _p ,V,W))$ is a consequence of the Weil conjectures/weight monodromy for abelian varieties (i.e. the fact that $H^0 (K_v ,V)=H^0 (K_v ,V^* \otimes W)=0$).
\end{proof}

\subsection{$p$-adic heights}
The $p$-adic height pairing, for an abelian variety $A$ over a number field $K$, can be described as a function
\[
h:A(K)\times A(K)\to \Q _p
\]
depending on a splitting of the Hodge filtration of $H^1 _{\dR}(A/K_{\mathfrak{p}})$ for every prime above $p$, and depending linearly on an idele class character, which may be thought of as a functional
\[
\oplus H^1 _g (K_v ,\Q _p (1))/H^1 _g (K ,\Q _p (1))\to \Q _p .
\]
As the dependence on the functional is linear, it is convenient to take the $p$-adic height (now only depending on the splittings) to simply be a map
\[
h:A(K)\times A(K)\to \oplus H^1 _g (K_v ,\Q _p (1))/H^1 _g (K ,\Q _p (1)) .
\]
In Nekov\'a\v r's construction \cite{nekovar}, this is in fact induced by a map
\[
H^1 _f (K ,V)\times H^1 _f (K ,V) \to \oplus  H^1 _g (K_v ,\Q _p (1))/H^1 _g (K ,\Q _p (1)) .
\]
which we will also call $h$. The original $h$ is then just the composite of the new $h$ with a product of Kummer maps 
\[
\kappa :A(K)\to H^1 _f (K ,V).
\] 

Nekov\'a\v r's height function on Bloch--Kato Selmer groups arises as a sum of local height functions on nonabelian cohomology varieties (although the original construction is not phrased in that way). Namely, for each prime $v$ we define a function
\[
h_v :M_g (K_v ;\Q _p ,V,\Q _p (1))\to H^1 _g (K _v ,\Q _p (1))
\]
(a more general construction will be described later) with the following two properties with respect to the biextension structure of mixed extensions:
\begin{itemize}
\item $h_v (M+c)=h_v (M)+c$ for all $c\in H^1 _g (K _v ,\Q _p (1))$.
\item If $M_1 $ and $M_2 $ have the same image in $H^1 _g (\Q _p ,V)$ under $\pi _{i*}$ (for $i=1$ or $2$), then $h_v (M_1 +_i M_2 )=h_v (M_1 )+h_v (M_2 )$.
\end{itemize}
Then define the global height of a mixed extension $M$ to be the image of $(h_v (M))_v $ in $\oplus H^1 _g (K _v ,\Q _p (1))/H^1 _g (K ,\Q _p (1))$. By the exact sequence of pointed sets
\[
H^1 _g (K ,\Q _p (1))\to M_g (K;\Q _p ,V,\Q _p (1))\to H^1 _g (K ,V)\times H^1 _g (K,V)\to 1,
\] 
the first property implies that the global height defines a function
\[
H^1 _g (K,V)\times H^1 _g (K ,V)\to \oplus H^1 _g (K _v ,\Q _p (1))/H^1 _g (K ,\Q _p (1)),
\]
($H^1 _g (K,\Q _p (1))$ acts transitively on the fibres of $M_g (K;\Q _p ,V,\Q _p (1))\to H^1 _g (K ,V)\times H^1 _g (K,V)$) 
and the second property implies it is bilinear.
\subsection{Generalised heights}
In \cite{QC2}, it was observed that this construction admits a natural generalisation, relevant to quadratic Chabauty. Let $W$ be a Galois stable quotient \footnote{later, we will also consider general quotients of $V^{\otimes 2}$. By Lemma \ref{lemma:sym_boring}, the difference is not significant for nonabelian Chabauty.} of $\Ker (\wedge ^2 H^1 _{\et }(X,\Q _p ) \stackrel{\cup }{\longrightarrow } H^2 _{\et }(X,\Q _p ))^* $. The local generalised height is a function
\[
h_v :M_{g,v}(\Q _p ,V,W)\to H^1 _g (\Q _v ,W)
\]
at each prime $v$ of $\Q $. It can be defined by the following properties:
\begin{enumerate}
\item It is equivariant with respect to the natural action of $(\Q _p ^\times )^3 \subset \Aut (\Q _p )\times \Aut (V)\times  \Aut (W)$ on both sides.
\item It is equivariant with respect to the natural action of $H^1 _g (\Q _v ,W)$, i.e.
\[
h_v (M+c)=c+h_v (M)
\]
for all $M\in M_{g,v}(\Q _p ,V,W)$ and $c\in H^1 _g (\Q _v ,W)$.
\end{enumerate}
Since $M_g (\Q _v ;\Q _p ,V,W)$ is an $H^1 _g (\Q _v ,W)$-torsor over $H^1 _g (\Q _v ,V)\times H^1 _g (\Q _v ,V^* \otimes W)$, we can think of a local generalised height as giving a trivialisation of this torsor. The difference of two choices of local generalised height defines a function
\begin{equation}\label{eqn:alg}
h_v -h_v ' :H^1 _g (K _v ,V)\times H^1 _g (K _v ,V^* \otimes W)\to H^1 _g (K _v ,W)
\end{equation}
By Lemma \ref{lemma:zero}, this implies that the local generalised height function is unique when $v$ is prime to $p$ (and indeed is just given by the isomorphism defined in that Lemma). In the case where $v|p$ the choice is nontrivial. In fact we can pin it down more, using the condition of equivariance with respect to $(\Q _p ^\times )^3 $. If we take the difference of \textit{algebraic} local generalised heights (with respect to the scheme structure on $M _g (K_v ;\Q _p ,V,W)$ induced by the identification with $U(\Q _p ,D_{\dR}(V),D_{\dR}(W))/F^0 $) then $h_v -h'_v$ in fact comes from an element of the vector space
\[
H^1 _g (K _v ,V)^* \otimes H^1 _g (K _v ,V^* \otimes W)^* \otimes H^1 _g (K _v ,W).
\]
In \cite{QC2}, a recipe is given for constructing a local height at a prime $v$ above $p$ depending on a choice of a splitting of the Hodge filtration on $H^1 _{\dR}(X_{K_v }/K_v )$. For the purposes of this paper, the choice of local generalised heights will not be important.

Given a finite set of primes $S$ containing all primes above $p$ and primes of bad reduction, and a tuple $(h_v )$ of local generalised heights at each prime $v$ in $S$, we define the global generalised height $h$ to be the function
\begin{align*}
h: H^1 _{f,S} (K,U(\Q _p ,V,W))\to (\oplus _{v\in S}H^1 _g (K_v ,W))/\loc (H^1 _{f,S}(K,W)) \\
M\mapsto (h_v (\loc _v (M))) \mod \loc H^1 _{f,S} (K,W).
\end{align*}
Here $H^1 _{f,S} (K,U(\Q _p ,V,W)) \subset H^1 (K,U(\Q _p ,V,W))$ is the subset of equivalence classes of torsors whose image in $H^1 (K_v ,U(\Q _p ,V,W))$ is in $H^1 _g (K_v ,U(\Q _p ,V,W))$ for $v$ in $S$, and in $H^1 _f (K_v ,U(\Q _p ,V,W))$ for all other $v$. Under the identification \eqref{eqn:iso_coh}, this may also be identified with equivalence classes of mixed extensions which are de Rham and unramified at all primes outside $S$.
When it is prudent to emphasise the choice of $W$, we will write this function as $h_W$.

By construction, $h$ factors through the quotient $H^1 _g (G_{K,S},U(\Q _p ,V,W))/H^1 _g (G_{K,S},W)$, and hence can be viewed as being defined on the image of $H^1 _{f,S}(K,U(\Q _p ,V,W))$ in $H^1 _f (K,V)\times H^1 _f (K,V^* \otimes W)$. Recall from Lemma \ref{lemma:WM} that $H^1 _f (K,V)=H^1 _g (K,V)$, and the same for $V^* \otimes W$.

It follows from the definitions that the map $h$ is bilinear in the following sense \cite[Lemma 3.8]{QC2}. Extend $h$ to $\Q _p [M_g (\Q _p ,V,W)]$ linearly, i.e.
\[
h(\sum \lambda _i [M_i ]):=\sum \lambda _i h(M_i ).
\]
Let $b$ denote the map
\begin{align*}
& \Q _p [M_g (\Q _p ,V,W)]\to H^1 _g (K,V)\otimes H^1 _g (K,V^* \otimes W) \\
& \sum \lambda _i [M_i ]\mapsto \sum \lambda _i \pi _{1 *}(M_i )\otimes \pi _{2*}(M_i ).
\end{align*}
Then 
\[
\Ker (h)\supset \Ker (b).
\]
For $h$ to be defined on the whole of $H^1 _g (K,V)\otimes H^1 _g (K,V^* \otimes W)$, it is enough for $b$ to be surjective.
Standard conjectures in arithmetic geometry imply that $h$ is in fact a pairing. One such conjecture is the following part of the Bloch--Kato conjectures.
\begin{conjecture}[Bloch--Kato, \cite{BK}, Conjecture 5.3 (i)]\label{conj:BK1}
Let $Z$ be a smooth proper scheme over $K$, and $r,m\in \Z $. Then the \'etale regulator
\[
K^{(r)}_{2r-1-m}(Z)\otimes \Q _p \simeq H^1 _g (K,H^m _{\et }(Z_{\overline{K}},\Q _p (r))).
\]
\end{conjecture}
\begin{lemma}
Conjecture \ref{conj:BK1} implies that 
\begin{equation}\label{eqn:surj}
(\pi _{1 *},\pi _{2*}):\Q _p [M_g (\Q _p  ,V,W)]\to H^1 _g (K  ,V)\times H^1 _g (K ,V^* \otimes W)
\end{equation}
is surjective.
\end{lemma}
\begin{proof}
Applying this conjecture with $Z=X^2 $, $m=2$ and $r=1$ gives that $H^1 _g  (K,H^2 _{\et }(Z_{\overline{K}},\Q _p (1)))=0$ (since negative $K$-groups are zero). Since $\wedge ^2 V(-1)$ is a direct summand of $H^2 _{\et }(Z_{\overline{K}},\Q _p (1))$, we deduce that $H^1 _g  (K,V^{\otimes 2}(-1))=0$. By Faltings' semisimplicity theorem, $W$ is a direct summand of $V^{\otimes 2}$, and hence $H^1 _g  (K,W(-1))=0$. In particular, for any finite set of primes $T$, the localisation map
\[
H^1 (G_{K,T},W^* (1))\to \oplus _{v\in T}H^1 (G_{K_v },W^* (1))
\]
is injective. By Poitou--Tate duality this implies that the localisation map
\begin{equation}\label{eqn:inj}
H^2 (G_{K,T},W)\to \oplus _{v\in T}H^2 (G_{K_v },W)
\end{equation}
is injective. Surjectivity of \eqref{eqn:surj} then follows from local surjectivity, which is proved in Lemma \ref{lemma:SES}. We have already seen that, for all primes $v$, the map
\[
M_g (K_v ;\Q _p ,V,W)\to H^1 _g (K_v ,V)\times H^1 _g (K_v ,V^* \otimes W),
\]
which by exactness implies that the boundary map
\[
H^1 _g (K_v ,V)\times H^1 _g (K_v ,V^* \otimes W)\to H^2 (K_v ,W)
\]
is zero. For any finite set $T$ of primes containing primes above $p$ and primes of bad reduction, we have a commutative diagram with exact rows
\begin{tikzcd}
M_{f,S}(K;\Q _p ,V,W) \arrow[d] \arrow[r] & H^1 _f (K,V)\times H^1 _f (K,V^* \otimes W) \arrow[d] \arrow[r] & H^2 (K,W) \arrow[d] \\
\prod _{v\in T}M_{g}(K_v ;\Q _p ,V,W) \arrow[r] & H^1 _f (K_v ,V)\times H^1 _f (K_v ,V^* \otimes W) \arrow[r] & \prod _{v\in T} H^2 (K_v ,W) \\
\end{tikzcd} 
To show that an element of $ H^1 _f (K,V)\times H^1 _f (K,V^* \otimes W)$ lifts to $M_{f,S}(K;\Q _p ,V,W)$ is the same as showing that its image in $H^2 (K,W)$ is zero. By \eqref{eqn:inj}, this is the same as showing that it is zero at each prime of $T$, which follows from the fact that it lifts at each prime of $T$. 
\end{proof}
Hence conjecturally the generalised height in fact defines a pairing
\[
H^1 _g (K ,V)\times H^1 _g (K ,V^* \otimes W)\to \oplus _v H^1 _g (K _v ,W)/\loc H^1 _g (K ,W).
\]
If $R$ is a subring of $\End _{G_K }(V)$, then one can define the notion of an $R$-equivariant local generalised height, and an $R$-equivariant generalised height (see \cite[\S 4.1]{QC2}). Recall that this corresponds to generalised height pairings defined by $R$-equivariant splittings of the Hodge filtration on $V$. If $s$ is an $R$-equivariant splitting, $M$ is a mixed extension in $M_g (\Q _p ,V,W)$, and $\gamma $ is an element of $R^\times $, then
\[
h(M)=h(\gamma ^* M),
\]
where $\gamma ^* M $ is the mixed extension obtained by composing the isomorphism $M_1 /M_2 \simeq V$ with $\gamma $. Viewed as a pairing on $H^1 _f (K,V)\times H^1 _f (V^* \otimes W)$, we deduce that whenever $h(E_1 ,E_2 )$ is defined, we have
\[
h(E_1 ,E_2 )=h (\gamma ^{-1}E_1 ,\gamma E_2 ).
\]
Let $G:=R^\times $. Then we deduce that $h$ factors through the $G$-coinvariants of $H^1 _f (K,V)\otimes _{\Q _p } H^1 _f (K,V^* \otimes W)$, i.e. through the map
\[
H^1 _f (K,V)\otimes _{\Q _p } H^1 _f (K,V^* \otimes W) \to (H^1 _f (K,V)\otimes _{\Q _p } H^1 _f (K,V^* \otimes W)) \otimes _{\Q _p [G]}\Q _p .
\]
We have 
\[
H^1 _f (K,V)\otimes _{\Q _p } H^1 _f (K,V^* \otimes W) \otimes _{\Q _p [G]}\Q _p \simeq H^1 _f (K,V)\otimes _{\Q _p [G] } H^1 _f (K,V^* \otimes W) 
\]
(see e.g. \cite[equation (3.2)]{atiyah}). Since $R$ is a finite $\Q _p $ algebra the natural map $\Q _p [G] \to R$ is surjective (see e.g. \cite[Unit theorem]{lenstra}). Hence
\footnote{Here and it what follows, when we say that the generalised height factors through $H^1 _f (K,V)\otimes _R H^1 _f (K,V^* \otimes W)$ we do not mean that it is necessarily defined on the whole of $H^1 _f (K,V)\otimes _R H^1 _f (K,V^* \otimes W)$. We mean that it factors through the projection from its domain of definition $S\subset H^1 _f (K,V)\otimes _{\Q _p } H^1 _f (K,V^* \otimes W)$ to its image in $H^1 _f (K,V)\otimes _R H^1 _f (K,V^* \otimes W)$}
\[
H^1 _f (K,V)\otimes _{\Q _p [G] } H^1 _f (K,V^* \otimes W) \simeq H^1 _f (K,V)\otimes _{R} H^1 _f (K,V^* \otimes W) .
\]
As is shown in loc. cit., a generalised height will be $R$-equivariant if it is constructed with respect to $R$-equivariant splittings of the Hodge filtration. A particularly interesting case is that of mixed extensions with graded pieces $\Q _p ,V^{\oplus n},W$, for $W$ some quotient of $V^{\otimes 2}$. Then we have an obvious inclusion of $\Mat _n (\Q _p )$ in $\End _{G_K }(V^{\oplus n})$, and an isomorphism
\[
H^1 _f (K,V^{\oplus n})\otimes _{\Mat _n (\Q _p )}H^1 _f (K,(V^{\oplus n})^* \otimes W)\simeq H^1 _f (K,V)\otimes _{\Q _p }H^1 _f (K,V^* \otimes W).
\]
If $c_i$ and $d_i$ are elements of $H^1 _f (\Q ,V)$ and $H^1 _f (\Q ,V^* \otimes W)$ respectively, and $M\in M(\Q _p ,V^{\oplus n},W)$ is a mixed extension of $(c_1 ,\ldots ,c_n )$ and $(d_1 ,\ldots ,d_n )$ in the sense of Definition \ref{defn:mixex}, then we will write $h(\sum c_i \otimes d_i )$ for the $\Mat _n (\Q _p )$-equivariant height of $M$. We will say that the height pairing of an element $v$ of $H^1 _f (K,V)\otimes H^1 _f (K,V^* \otimes W)$ is defined if there is an $n$ such that $v$ is in the image of $H^1 _g (K,U(\Q _p ,V^{\oplus n},W))$.
\begin{example}
Given exact sequences of crystalline $G_K$-representations
\begin{align*}
0\to V\to E_1 \to \Q _p \to 0 \\
0\to V \to E_2 \to \Q _p \to 0 ,
\end{align*}
the tensor product $E_1 \otimes E_2 $ is naturally an object in $\mathcal{M}(\Q _p ,V^{\oplus 2},V^{\otimes 2})$. Note that $E_1 \otimes E_2 $ is a mixed extension of $(E_1 ,E_2 )$ by $(\iota _1 (E_2 ),\iota _2 (E_1 ))$, where $\iota _i $ are the two natural inclusions of $V$ into $\Hom (V,V^{\otimes 2})$. It follows that the generalised height of $E_1 \otimes \iota _1 (E_2 )+E_2 \otimes \iota _2  (E_1 )$ is always defined. Here $\iota _1 $ and $\iota _2$ denote the pushforward maps $H^1 _f (K,V)\to H^1 _f (K,V^* \otimes V^{\otimes 2})$ corresponding to left and right multiplication.
\end{example}
\begin{remark}
Without assuming the Bloch--Kato conjectures, the generalised height pairing is defined on the kernel of the cup product map
\[
\cup :H^1 _g (K ,V)\times H^1 _g (K ,V^* \otimes W)\to H^2 (K ,W).
\]
Moreover this map lands in the kernel of 
\[
\loc :H^2 (K ,W)\to \oplus _v H^2 (K _v  ,W).
\] 
On the other hand let $H^2 _c (G_{K,S} ,W)$ denote compactly supported Galois cohomology, i.e. the second cohomology of the (shifted) cone $C^\bullet _c (G_{K ,S},W)$ of
\[
R\Gamma (G_{K ,S},W)\to \oplus _{v\in S}R\Gamma (G_{K _v },W).
\]
Then we have a short exact sequence
\[
0\to \frac{\oplus _{v\in S}H^1 (K _v ,W)}{\loc H^1 (G_{K,S} ,W)}\to H^2 _c (G_{K,S} ,W)\to \Ker (H^2 (G_{K,S} ,W)\to \oplus _{v\in S} H^2 (K _v ,W))\to 0.
\]
Hence it would be natural to have a generalised height pairing with values in $H^2 _c (G_{K,S} ,W)$. 
\end{remark}
\subsection{The mixed extensions associated to fundamental groups}
We first recall a few properties of the mixed extensions associated to fundamental groups (see \cite[\S 3]{DG}). Let $X$ be a smooth geometrically connected variety over a field $K$ of characteristic zero and $b$ and $z$ points in $X(K)$. Let $E_n (X,b)$ be the $\Q _p $-\'etale sheaf on $X_{\overline{K}}$ characterised by either of the equivalent properties:
\begin{itemize}
\item It corresponds to the $\pi _1 (X_{\overline{K}},b)$-module given by the maximal $n$-nilpotent quotient of the universal enveloping algebra of the Lie algebra of the $\Q _p $-Malcev completion of $\pi ^{\et }_1 (X_{\overline{K}},b)$.
\item $E_n (X,b)$, together with a distinguished point $1$ in the fibre at $b$, is universal for pointed $n$-unipotent $\Q _p $-\'etale local systems on $X_{\overline{K}}$, i.e. for every finite dimensional $\Q _p $-representation $M$ of $\pi _1 ^{\et }(X_{\overline{K}},b)$ admitting a $\pi _1 $-stable filtration
\[
M=M_0 \supset M_1 \supset \ldots \supset M_n =0
\]
such that $\pi _1 $ acts trivially on $\gr (M)$, and every $m\in M$, there is a unique morphism of $\pi _1 $-representations $b^* E_n (X,b)\to M$ sending $1$ to $m$.
\end{itemize}
Note that the $\pi _1 (X_{\overline{K}},b)$-module corresponding to $E_n (X,b)$, thought of as a $\pi _1 (X(\mathbb{C}),b)$-module via an embedding $\overline{K} \to \mathbb{C}$, is simply 
\[
\Z [\pi _1 (X(\mathbb{C}),b)]/I^{n+1} \otimes _{\Z }\Q _{p },
\]
where $I\subset \Z [\pi _1 (X(\mathbb{C}),b)]$ is the kernel of the map
\begin{align*}
\Z [\pi _1 (X(\mathbb{C}),b)]\to \Z \\
\gamma \mapsto 1.
\end{align*}
Let $E_n (X;b,z)$ denote the fibre of $E_n (X,b)$ at $z$. By the above, an element of $E_n (X;b,z)$ can be thought of as a $\Q _{p }$-linear combination of homotopy classes of paths from $b$ to $z$ (although this perspective will not be used in what follows). $E_n (X;b,z)$ inherits an action of the Galois group of $K$. By the above, $E_n (X,b)$ admits a filtration corresponding to the filtration by powers $I^k$ of the augmentation ideal on the universal enveloping algebra. We denote by $I^k E_n (X;b,z)$ the fibre at $z$ of the local system $I^k E_n (X,b)$. Where there is no possibility of confusion we write the fibre at $b$ as simply $E_n (X,b), I^k /I^j (b)$, etc. When $n=2$, we will drop the subscript $n$ and denote this simply as $E(X,b),E(X;b,z)$, etc. Finally, when there is no risk of ambiguity over the choice of $X$ we will write this simply as $E(b), E(b,z)$, etc.

In this paper we will be interested in this construction when $n=2$. In thie case, the graded pieces of $E_2 (X;b,z)$ are $\Q _p ,V$ and $V^{\otimes 2}/\Q _p (1)$ (see e.g. \cite[\S 3.3]{QC2}). First we recall Beilinson's cohomological description of $E(X;b,z)$.
\begin{theorem}[\cite{DG}, Proposition 3.4]
We have an isomorphism of $\Q _p$-representations of $\Gal (K)$
\[
E(X;b,z)\simeq H^2 _{\et }(X^2 ;A\cup B\cup \Delta ,\Q _p )^* ,
\]
where $A=\{ b\} \times X$ and $B=X\times \{ z\}$.
\end{theorem}
Hain and Matsumoto \cite[Theorem 3]{HM} relate the extension class of $I/I^3 $ to the Ceresa cycle, and hence to the Gross--Schoen cycle \cite{GS}. Such a relation also follows from Beilinson's theorem. We will not need such a relation in what follows, and for the purposes of this paper one may take $\Delta (b)$ in the statement of Theorem \ref{thm:only_one} to be the image of $I/I^3 (b)$ in $\Ext ^1 _{\Gal (K)}(V,V^{\otimes 2}/\Q _p (1))$.
Although the the extension class $IE(x,z)$ depends on $x$ and $z$, its dependence is abelian, and easy to describe.
\begin{lemma}\label{lemma:IE}
Let $i_1 $ and $i_2$ be the inclusions $V\hookrightarrow V^* \otimes (V^{\otimes 2}/\Q _p (1))$ corresponding to left and right multiplication. Then we have the following identities in $\Ext ^1 _{\Gal (K)}(V,V^{\otimes 2}/\Q _p (1))$.
\begin{enumerate}
\item 
\[
[IE(x,z)]=[IE(x,y)]+i_{1*}[E_1 (y,z)].
\]
\item
\[
[IE(y,z)]=[IE(x,z)]+i_{2*}[E_1 (y,x)].
\]
\end{enumerate}
\end{lemma}
\begin{proof}
Recall (see e.g. \cite[\S 10]{deligne}) that $E_n (x,y)$ is the fibre at $x,y$ of the local system on $X_{\overline{K}}\times X_{\overline{K}}$ corresponding to $E_n (b,b)$, viewed as a $\pi _1 ^{\et } (X_{\overline{K}}\times X_{\overline{K}}; b,b)$ via the action $(\gamma _1 ,\gamma _2 )\cdot x=\gamma _1 \cdot x\cdot \gamma _2 ^{-1}$. Then $IE(x,y)$ inherits the structure of the fibre at $(x,y)$ of an abelian $\Q _p$-local system on $X\times X$ corresponding to the $\pi _1 (X\times X)^{\ab}$-module structure on $\Z [\pi _1 (X)]/I^2 $ given by 
\[
(\gamma _1 ,\gamma _2 )\cdot \gamma _3 :=\gamma _1 \gamma _3 \gamma _2 ^{-1}.
\]
\end{proof}

\begin{lemma}\label{lemma:W}
If $X$ is hyperelliptic, and $P$ is a Weierstrass point, then $[IE(P,P)]=0$ in $\Ext ^1 _{\Gal (K)}(V,V^{\otimes 2}/\Q _p (1))$.
\end{lemma}
\begin{proof}
Let $w$ denote the hyperelliptic involution. Then by functoriality
\[
[IE(P,P)]=w^* [IE(P,P)]
\]
On the other hand $w^*$ acts as $-1$ on $V$ and as $1$ on $V^{\otimes 2}/\Q _p (1)$, so we deduce that $2[IE(P,P)]=0$.
\end{proof}
\begin{lemma}
If $X$ is a hyperelliptic curve, then $E(b,z)$ is an extension of $\kappa ([z]-[b])$ in $H^1 _f (K ,V)$ by $i_{1*} (z-W)+i_{2*}  (W-b)$ in $H^1 _f (K ,V^{\otimes 3}(-1)/V)$, where $i_1$ and $i_2 $ are the maps
\[
V\to \Hom (V,V^{\otimes 2}/\Q _p (1))
\]
given by left and right multiplication respectively, and $W=\frac{1}{2g+2}D$ and $D$ is the divisor of Weierstrass points.
\end{lemma}
\begin{proof}
This follows from Lemma \ref{lemma:IE} and Lemma \ref{lemma:W}.
\end{proof}
In our study of mixed extensions associated to $2$-nilpotent fundamental groups, we may choose $W$ to be any direct summand of $V^{\otimes 2}/\Q _p (1)$. We define $E_W (b,z)$ to be the mixed extension with graded pieces $\Q _p ,V$ and $W$ obtained by quotienting $E(b,z)$ by $\Ker (V^{\otimes 2}/\Q _p (1)\to W)$. We have a direct sum decomposition 
\[
V^{\otimes 2}/\Q _p (1)\simeq (\wedge ^2 V/\Q _p (1))\oplus \Sym ^2 V.
\]
\begin{lemma}\label{lemma:sym_boring}
We have an isomorphism of mixed extensions
\[
E(b,z)/(\wedge ^2 V/\Q _p (1))\simeq \Sym ^2 E_1 (b,z).
\]
\end{lemma}
\begin{proof}
This follows from the Galois equivariant isomorphism of $\pi _1 (X_{\overline{K}},b)$-representations $E(b)/(\wedge ^2 V/\Q _p (1))\simeq \Sym ^2 E_1 (b)$.
\end{proof}

It follows from the definitions that the generalised height function satisfies Galois descent, i.e. if $L|K$ is a finite Galois extension of number fields the diagram 
\[
\begin{tikzcd}
\Q _p [H^1 _g (L,U(\Q _p ,V,W))/H^1 _g (L,W)]^{\Gal (L|K)} \arrow[r]           & \big( (\oplus _{w}H^1 _g (L_w ,W))/\loc H^1 _g (L ,W) \big) ^{\Gal (L|K)}           \\
\Q _p [H^1 _g (K,U(\Q _p ,V,W))/H^1 _g (K,W)] \arrow[r]   \arrow[u]        & (\oplus _{v}H^1 _g (K_v ,W))/\loc H^1 _g (K ,W)   \arrow[u, "\simeq "] \\
\end{tikzcd}
\]
commutes. Using the right-hand isomorphism, we can think of the generalised height as defining a function
\begin{align*}
& \colim _{L|K }\mathrm{Im}\left( \Q _p [H^1 _g (L,U(\Q _p,V,W)]^{\Gal (L|K )} \to \colim _{L|K}(H^1 _f (L ,V)\otimes H^1 _f (L,V^* \otimes W))^{\Gal (L|K )} \right)  \\
& \to (\oplus _v H^1 _g (K _v ,W)/H^1 _g (K ,W),
\end{align*}
where the colimit is over finite Galois extensions $L|K $. We now consider how to define height functions on divisors, which could be viewed as generalisations of the usual quadratic $p$-adic height function on an elliptic curve.

\begin{lemma}
If $D=\sum n_i z_i $ is a principal divisor on a curve $X$, then for any basepoints $x,y$,
\[
\sum n_i h(E_W (x,z_i ))=\sum n_i h(E_W (y,z_i )).
\]
\end{lemma}
\begin{proof}
It is enough to prove that
\[
\sum _i n_i [\kappa ([x]-[z_i ])]\otimes [IE_W (x,z_i )]=\sum _i n_i [\kappa ([y]-[z_i ] )]\otimes [IE_W (y,z_i )],
\]
which follows from Lemma \ref{lemma:IE}.
\end{proof}
We will henceforth simply write $h_W (D)$, or $h_{X,W}(D)$, or $h_X (D)$, for the $W$-generalised height associated to a principal divisor $D$.
\subsection{Mixed extensions from curves in $X\times X$}\label{subsec:curve}
Fix a closed immersion $(f_1 ,f_2 ):C\hookrightarrow X\times X$, where $C$ is a geometrically irreducible curve whose Jacobian surjects onto $\Jac (X)^2 $ via the map induced by the $f_i$. We choose a direct sum decomposition
\[
V_C \simeq V_A \oplus V_B
\]
where $V_C$ is the $\Q _p$-Tate module of the Jacobian of $C$, $V_A$ is the $\Q _p$-Tate module of the connected component of the identity $A$ of 
\[
\Ker (\Jac (C)\to \Jac (X)^2 )
\]
induced by $(f_1 ,f_2 )$, and $V_B $ is the $\Q _p $-Tate module of $B:=\Jac (C)/A$.
We obtain a surjection
\[
V_C ^{\otimes 2} \to \wedge ^2 V.
\]
induced by the two maps from $V_C$ to $V$. 
Hence for any pair of points $x,y\in C(K)$, we obtain a mixed extension $E_C (x,y) \in \mathcal{M}(\Q _p ,V_C ,W)$, where $W:=\wedge ^2 V/(\NS (\Jac (X)_{\overline{\Q }})^* \otimes \Q _p (1))$. Recall that, by Faltings' proof of the Tate conjecture for $H^2 $ of an abelian variety over a number field, $\NS (\Jac (X))^* \otimes \Q _p (1)$ may be viewed as a direct summand of $\wedge ^2 V$, and that it necessarily contains the image of $\Ker (\wedge ^2 V_C \to \gr _2 E_2 (C))$. Now let $R:=\End (B)\otimes \Q _p \times \End (A) \otimes \Q _p \subset \End _{\Gal (K)}(V_C )$, let $R_0 :=\End (\Jac (X))\otimes \Q _p $ and $R_1 :=\End (A)\otimes \Q _p $. Then any associated $R$-equivariant generalised height defines a map which factors through
\begin{align*}
& H^1 _f (\Q ,V_C)\otimes _R H^1 _f (\Q ,V_C ^* \otimes W)
\simeq & H^1 _f (\Q ,V)\otimes _{R_0 } H^1 _f (\Q ,V^* \otimes W)\oplus H^1 _f (\Q ,V_A )\otimes _{R_1 }H^1 _f (\Q ,V_A ^* \otimes W).
\end{align*}
We explain this isomorphism in a more general setting.
\begin{lemma}
\begin{enumerate}
\item Given $K$-algebras $R$ and $S$, $R$-modules $M_1 $ and $M_2 $ and $S$-modules $N_1 $ and $N_2$, we have an isomorphism
\[
(M_1 \oplus N_1 )\otimes _{R\times S}(M_2 \oplus N_2 ) \simeq M_1 \otimes _R M_2 \oplus N_1 \otimes _S N_2 
\]
(where we view $M_i$ as an $R\times S$-module in the obvious way, and the same for the $N_i $).
\item Given $R$-modules $M_1 $ and $M_2 $,
\[
M_1 ^{\oplus n}\otimes _{\Mat _n (R)}M_2 ^{\oplus n}\simeq M_1 \otimes _R M_2 .
\]
\end{enumerate}
\end{lemma}
\begin{proof}
For part (1), the natural map $(M_1 \oplus N_1 )\times (M_2 \oplus N_2 )\to M_1 \otimes M_2 \oplus N_1 \otimes N_2 $ satisfies the universal property for the tensor product, because $M_1 \otimes N_2 =N_1 \otimes M_2 =0$ for example by a simple argument considering the action of $R\times \{ 0 \}$ and $\{ 0 \} \times S$. For part (2), similarly the map $M_1 ^{\oplus n}\times M_2 ^{\oplus n}\to M_1 \otimes _R M_2 $ sending $((m_i ),(m' _i ))$ to $\sum m_i \otimes m'_i$ satisfies the universal property for $\Mat _n (R)$-bilinear maps.
\end{proof}
\begin{lemma}\label{lemma:EC}
The class of $E_C (x,y)$ in $H^1 _f (\Q ,V_A ^* \otimes W)$ is independent of $x$ and $y$.
\end{lemma}
\begin{proof}
The class of $E_C (x,y)$ in $H^1 _f (\Q ,V_A ^* \otimes W)$ is the image of the class $[IE(x,y)]$ in $H^1 _f (\Q ,V_C ^* \otimes (\wedge ^2 V_C /\Q _p (1))$ under the quotient map
\[
H^1 _f (\Q ,V_C ^* \otimes (\wedge ^2 V_C /\Q _p (1))\to H^1 _f (\Q ,V_A ^* \otimes W)
\]
This follows from Lemma \ref{lemma:IE}.
\end{proof}
Hence the map
\[
\Q _p [C(\overline{\Q})\times C(\overline{\Q })]^{\Gal (\overline{\Q }|\Q )}\to \colim _{K|\Q }(H^1 _f (K,V_C)\otimes _R H^1 _f (K,V_C ^* \otimes W))^{\Gal (K|\Q )}
\]
induces a map
\[
\theta :\PDiv ^0 (C)\to \colim _{K|\Q }(H^1 _f (K,V)\otimes H^1 _f (K,V^* \otimes W))^{\Gal (K|\Q )}.
\]
We deduce the following.
\begin{lemma}\label{lemma:find_div}
Let $C\subset X\times X$ be a geometrically irreducible curve, $W$ as above, and $D=\sum n_i (P_i ,Q_i )$ a principal divisor on $C$. Let $h_C$ denote the generalised height on $M _g (\Q _p ,V_C ,W)$ with respect to an $R$-equivariant splitting $s$ of the Hodge filtration as above. Let $h$ denote the generalised height on $X$ with respect to a splitting compatible with $s$. Suppose the class of $\sum n_i P_i \otimes Q _i $ in $H^1 _f (K,V)^{\otimes 2}\subset H^1 _f (K,V)\otimes H^1 _f (K,V^* \otimes W)$ is in the image of $\colim \Q _p [H^1 _g (K,U(\Q _p ,V,W))]^{\Gal (K|\Q )}$. 
Then
\[
h_C (D)=h(\sum n_i P_i \otimes Q _i )
\]
where the height on the right hand side is the generalised height on the image of
\[
\colim \Q _p [H^1 _g (K,U(\Q _p ,V,W))]^{\Gal (K|\Q )}\to \colim \left( H^1 _f (K,V)\otimes H^1 _f (K,V^* \otimes W) \right) ^{\Gal (K|\Q )}  .
\]
\end{lemma}
\begin{proof}
For all primes $v$, if $M$ is a mixed extension in $H^1 _g (K_v ,U(\Q _p ,V,W))$ then by construction we have
\[
h_{v}(M)=h_{C,v}(M\oplus T )
\]
where $T:=\Ker (V_C \stackrel{f_{1 *}}{\longrightarrow }V)$, and $M\oplus T$ is viewed as an element of $H^1 _g (K_v ,U(\Q _p ,V_C ,W))$ in the obvious way.

On the other hand, we have a commutative diagram
\begin{small}
\[
\begin{tikzcd}
\Q _p [H^1 _g (K,U(\Q _p ,V,W))] \arrow[d] \arrow[r] & \Q _p [H^1 _g (K,U(\Q _p ,V_C ,W))]  \arrow[d] \\
H^1 _f (K,V)\otimes _{R_0 }H^1 _f (K,V^* \otimes W) \arrow[r]           & H^1 _f (K,V)\otimes _{R_0 } H^1 _f (K,(V^* \otimes W) \oplus H^1 _f (K,V_A )\otimes _{R_1 } H^1 _f (K,V_A ^* \otimes W).          
\end{tikzcd}
\]
\end{small}
Hence the lemma follows from Lemma \ref{lemma:EC}.
\end{proof}
\begin{remark}
Note that the cohomological interpretation of the generalised heights constructed from the curves $C$ is to look at the Galois representations associated to appropriate subquotients of $H^2 (X^2 -C \cup \{P_i \} \times X \cup X\times \{Q _j \})$. This is one motivation for the constructions in the next section.
\end{remark}
\subsection{Generalised heights and quadratic Chabauty}\label{ludtke}
We briefly recall (and correct!) the relation between generalised heights and quadratic Chabauty described in \cite{QC2}. The map
\begin{align*}
X(K)\to M_g (\Gal (K);\Q _p ,V,V^{\otimes 2}/\Q _p (1))
\end{align*}
defined by sending $z\in X(K)$ to $E_2 (X;b,z)$ factors through Kim's nonabelian Kummer map
\[
j_2 :X(K)\to H^1 (\Gal (K),U_2 )
\]
via the map
\[
H^1 (\Gal (K),U_2 )\to H^1 (\Gal (K),U(\Q _p ,V,V^{\otimes 2}/\Q _p (1)))
\]
defined as follows. The action of $U_2 $ on the enveloping algebra of its Lie algebra induces a homomorphism from $U_2 $ to the group $\Aut ^{\un }(E_2 (X,b))$ of unipotent automorphisms of the filtered algebra $E_2 (X;b)$. This gives a map
\[
H^1 (\Gal (K),U_2 )\to H^1 (\Gal (K),\Aut ^{\un }(E_2 (X,b))).
\]
The set of unipotent isomorphisms from $E_2 (X,b)$ to its associated graded $\Q _p \oplus V\oplus V^{\otimes 2}/\Q _p (1)$ has the structure of a Galois equivariant \\
$(U(\Q _p ,V,V^{\otimes 2}/\Q _p (1)),\Aut ^{\un }(E_2 (X,b)))$-bitorsor, giving an isomorphism
\[
H^1 (K,U(\Q _p ,V,V^{\otimes 2}/\Q _p (1))\simeq H^1 (K,\Aut ^{\un }(E_2 (X;b,z))).
\]
hence we obtain a map from $H^1 (K,U_2 )$ to $H^1 (K,U(\Q _p ,V,V^{\otimes 2}/\Q _p (1)))$ as desired. In the language of mixed extensions, this sends a $U_2 $-torsor $P$ to the twist of $E_2 (X,b)$ by $P$.

We obtain a commutative diagram
\[
\begin{tikzcd}
X(\Q ) \arrow[d] \arrow[r] & H^1 _g (\Q ,U_2 ) \arrow[d] \arrow[r] & H^1 _g (\Q ,U(\Q _p ,V,V^{\otimes 2}/\Q _p (1))) \arrow[d] \\
\prod X(\Q _v ) \arrow[r]           & \prod H^1 _g (\Q _v ,U_2 ) \arrow[r]            & \prod H^1 _g (\Q _v ,U(\Q _p ,V,V^{\otimes 2}/\Q _p (1)))          
\end{tikzcd}
\]
Recall the set $X(\Q _p )_2 \subset X(\Q _p )$ is the projection to $X(\Q _p )$ of the subset of $\prod _v X(\Q _v )$ whose image in $\prod H^1 _g (\Q _v ,U_2 )$ lies in the (scheme-theoretic) image of $H^1 _g (\Q ,U_2 )$. The maps
\[
H^1 _g (K,U_2 )\to M_g (K_v ;\Q _p ,V,V^{\otimes 2}/\Q _p (1))
\]
for $K=\Q $ or $\Q _v$ are injective. In \cite[Proof of Lemma 4.1]{QC2}, we erroneously claimed that this implies that $P\in H^1 _f (\Q _p ,U_2 )$ is in the image of $H^1 _g (\Q ,U_2 )$ if and only if $E(X,b)^{(P)}\in M_g (\Q _p ;\Q _p ,V,V^{\otimes 2}/\Q _p (1))$ is in the image of $M_g (\Q ;\Q _p ,V,V^{\otimes 2}/\Q _p (1))$. Although we do not know of a counterexample, there is no reason that this should be true, as if the dimension of $H^1 _f (\Q ,V^{\otimes 3}(-1)/V)$ is very large then it could happen that $\dim H^1 _f (\Q ,U_2 )<\dim H^1 _f (\Q _p ,U_2 )$ but $M_g (\Q ;\Q _p ,V,V^{\otimes 2}/\Q _p (1))$ surjects onto $M_g (\Q _p ;\Q _p ,V,V^{\otimes 2}/\Q _p (1))$. The correct statement is that $P$ is in the image of $H^1 _g (\Q ,U_2 )$ if and only if $E(X,b)^{(P)}$ comes from a global mixed extension in the image of $H^1 _g (\Q ,U_2 )$. I am grateful to Martin L\"udtke for pointing this out.
\section{The Albanese kernel and mixed extensions from algebraic cycles}\label{sec:alb}
Let $K$ be a field and $Z/K$ a smooth projective irreducible variety of dimension $d$. The Albanese kernel is defined to be
\[
T(Z):=\Ker (\CH^d (Z)_0 \to \Alb (Z)).
\]
In the case where $K$ is a number field, (a special case of) the Beilinson--Bloch conjectures may be stated as follows.
\begin{conjecture}[\cite{beilinson}]\label{conj:BB}
$T(Z)$ is a finite group.
\end{conjecture}

We are mostly interested in the case $Z=X\times X$, where $X/\Q $ is a curve of genus $g>1$, and describe applications of this conjecture to the determination of rational points on $X$. This idea of using $\CH ^2 (X^2 )$ to study $X(\Q )$ also appears in work of Esnault and Wittenberg \cite{EW}, where they construct a class in $\CH ^2 (X^2 )$ giving an obstruction to an $X$ with a divisor of degree one containing a rational point, refining the obstruction coming from the $2$-nilpotent fundamental group. Although this work was an inspiration for the present paper, our setting is simpler and our results more modest: we assume that we have a rational point, and seek to use Conjecture \ref{conj:BB} to determine $X(\Q )$ via the Chabauty--Kim method \cite{kim:chabauty} \cite{kim:siegel}.

We have a map \cite{somekawa} \cite{RS}
\[
\Div (X)\times \Div (X)\to Z^2 (X^2 )
\]
given by
\[
(\sum n_i P_i ,\sum m_j Q_j )\mapsto \sum n_i m_j (P_i ,Q_j ).
\]
We write the image of $(D_1 ,D_2 )$ as $D_1 \boxtimes D_2 $.
This induces a bilinear map (the Somekawa cup product)
\[
\Pic ^0 (X)\times \Pic ^0 (X)\to T(X\times X),
\]
More generally, for any finite Galois extension $L|K$,  this defines a map
\[
(\Pic ^0 (X)(L)^{\otimes 2} )^{\Gal (L|K)}\to T(X\times X)(K).
\]
which gives an isomorphism \cite{RS}
\[
\colim _{L|K}(\Pic ^0 (X)(L)^{\otimes 2} )^{\Gal (L|K)}\stackrel{\simeq }{\longrightarrow} T(X\times X)(K).
\]
\begin{lemma}\label{lemma:only_one}
Assuming the Beilinson--Bloch conjecture, there is an algorithm which, given a pair of points $P_1 , P_2 \in \Jac (X)$ (represented by divisors $D_1 $ and $D_2$), returns a positive integer $n$, a finite set of curves $\iota _i :C_i \hookrightarrow X_{\overline{\Q }}\times X_{\overline{\Q }}$, and principal divisors $\divv (f_i ) $ on $C_i$, such that the generalised height pairing $h(P_1 ,P_2 )$ is equal to
\[
\frac{1}{n}\sum h_{\iota _i }(\divv (f_i ) ).
\]
\end{lemma}
\begin{proof}
Beilinson--Bloch implies that we can write $D_1 \boxtimes D_2 $ as $\frac{1}{n}\sum _i \divv _{C_i }(f_i )$ for some $n$, $C_i$ and $f_i$ as above. Note that, as a consequence of this, there is an algorithm to find such $C_i $ and $f_i$, for the simple reason that the collection of all possible $C_i $ and $f_i$ is countable.  Hence the Lemma follows from Lemma \ref{lemma:find_div}.
\end{proof}
We can now give the proof of Theorem \ref{thm:only_one} from the introduction.
\begin{proof}[Proof of Theorem \ref{thm:only_one}]
Let $D_1 =\sum m_i Q_i $.
By Lemma \ref{lemma:IE}, we have
\begin{align*}
h(P_1 ,\Delta (b)+P_2 ) & =\sum _i m_i h(P_i -b ,I/I^3 (b)+P_2 ) \\
& =\sum _i m_i h(P_i)+h(P_i -b,P_2 -P_i +b).
\end{align*}
By Lemma \ref{lemma:only_one}, Beilinson--Bloch implies we can write each $h(P_i -b,P_2 -P_i +b)$ as $\frac{1}{n_i }\sum _j h_{\iota _{ij} }(E_{ij})$ for some positive integers $n_i$, curves $\iota _{ij}:C_{ij}\hookrightarrow X^2 $ and principal divisors $E_{ij}$ on $C_{ij}$ such that $\sum _j E_{ij}= n_i (P_i -b)\boxtimes (P_2 -P_i +b)$.
\end{proof}
The map $\boxtimes $ has the following compatibility with the cup product in Galois cohomology.
\begin{lemma}
The diagram 
\[
\begin{tikzcd}
\CH ^1 (X)_0 \times \CH ^1 (X)_0 \arrow[r] \arrow[d] & T(X^2 ) \arrow[d] \\
H^1 (K,H^1 _{\et }(X,\Z _p (1)))\times H^1 (K,H^1 _{\et }(X,\Z _p (1))) \arrow[r] & H^2 (K,H^2 _{\et }(X^2 ,\Z _p (2)))
\end{tikzcd}
\]
commutes.
\end{lemma}
\begin{proof}
This is a consequence of the compatibility of the \'etale cycle class maps with cup products. This gives a commutative diagram
\[
\begin{tikzcd}
\CH ^1 (X) \times \CH ^1 (X) \arrow[r] \arrow[d] & \CH ^2 (X^2 ) \arrow[d] \\
H^2 _{\et } (X_K ,\Z _p (1)) \times H^2 _{\et } (X_K ,\Z _p (1)) \arrow[r] & H^4 _{\et } (X_K ,\Z _p (2)).
\end{tikzcd}
\]
Under the Grothendieck--Leray spectral sequence, the cup product in \'etale cohomology induces the cup product
\[
H^1 (K,H^1 _{\et }(X,\Z _p (1)))\times H^1 (K,H^1 _{\et }(X,\Z _p (1))) \to H^2 (K,H^2 _{\et }(X^2 ,\Z _p (2)))
\]
in Galois cohomology (see \cite[Proof of Proposition 2.4]{yamazaki} for a more general argument).
\end{proof}
In fact, more is true: by work of Kahn--Yamazaki \cite{KY} and Mochizuki \cite{mochizuki}, we can interpret the Somekawa product as a Yoneda product in the derived category of motives. In particular, if there was an underlying abelian category of mixed motives, the Somekawa product would be exactly the obstruction to concatenating the motivic extensions by the classes in $\CH ^1 (X)_0$.

More generally, for any two smooth projective irreducible varieties $X_1 $ and $X_2 $ of dimensions $d_1 $ and $d_2$ we can consider a product
\[
\CH ^{d_1 }(X_1 )_0 \times \CH ^{d_2 }(X_2 )_0 \to T(X_1 \times X_2 ).
\]
whose image should again be torsion by the Beilinson--Bloch conjectures.
\subsection{A space of motivic mixed extensions}
We now give a candidate for a motivic refinement of a nonabelian Galois cohomology variety. As is explained below, this should be seen as a generalisation of the \'etale regulator map
\[
\CH^2 (X^2 ,1)\to H^1 (K,H^2 _{\et }(X^2 _{\overline{K}},\Z _p (2)))
\]
and the description of $\CH ^2 (X^2 ,1)$ in terms of the Gersten resolution.
Define $\mathcal{M}(X\times X)$ to be the set of tuples $(Z_1 ,Z_2 ,\sum (C_i ,D_i ))$,  where

\begin{enumerate}
\item $Z_1 $ and $Z_2 $ are degree zero divisors on $X$.
\item The $C_i $ are a finite collection of curves on $X^2 $.
\item $D_i$ is a principal divisor on $C_i $ (so $\sum (C_i ,D_i )$ is an element of the direct sum, over $C\in (X^2 )^{(1)}$, of $\Div ^0 (C)$).
\item 
\[
Z_1 \boxtimes Z_2 =\sum _i D_i
\]
in $Z^2 (X^2 )$,
\end{enumerate}
modulo the equivalence relation generated by the following relations:
\begin{align*}
& (Z_1 ,Z_2 ,(C_i ,D_i ))\sim (Z_1 +\divv (g),Z_2 ,\sum (C_i ,D_i )+(\pi _2 ^* Z_2 ,\divv (g)) \\
& (Z_1 ,Z_2 ,(C_i ,D_i ))\sim (Z_1 ,Z_2+\divv (g) ,\sum (C_i ,D_i )+(\pi _1 ^* Z_1 ,\divv (g)) \\
& (Z_1 ,Z_2 ,(C_i ,D_i ))\sim (Z_1 ,Z_2 ,\sum (C_i ,D_i )+\partial (\{ g_1 ,g_2 \} )) 
\end{align*}
where $\partial \{ g_1 ,g_2 \} $ is the image of $\{ g_1 ,g_2 \} \in K^M _2 (K(X^2 ))$ in $\oplus _{C\in (X^2 )^{(1)}}\Div ^0 (C)$ under the residue map.
%
We obtain a map
\[
\mathcal{M}(X\times X)\to \CH^1 (X)_0 \times \CH^1 (X)_0
\]
which lands in the kernel of the map to $T(X\times X)$.
Let $|C|:=\cup C_i $. Let $N\subset V^{\otimes 2}$ be the the first step of the coniveau filtration, i.e. the subspace generated by $H^2 _C (X^2 ,\Q _p (2))$ for curves $C \subset X^2 $. Let $H^1 (K,U(\Q _p ,V,V^{\otimes 2})/N)_{s}\subset H^1 (K,U(\Q _p ,V,V^{\otimes 2}/N)$ denote the set of mixed extensions whose image in $H^1 (K,V^* \otimes V^{\otimes 2}/N)$ lies in the image of $H^1 (K,V)$ via the left multiplication map
\[
V\to V^* \otimes V^{\otimes 2}/N
\]
(note that via this map $V$ is a direct summand of $V^* \otimes V^{\otimes 2}/N$).

We recall that the the group $\CH ^2 (X^2 ,1)$ can be identified with the cohomology of the complex
\begin{equation}\label{eqn:gersten}
K^2 _M (X^2 )\to \oplus _{C\in (X^2 )^{(1)}}K(C)^\times \to \oplus _{P\in (X^2 )^{(2)}}\Z
\end{equation}
This identification is the composite of the isomorphism \cite[Theorem 2.5]{landsburg}
\[
\CH ^2 (X^2 ,1)\simeq H^1 (X^2 ,\mathcal{K}_2 )
\]
with the identification of $H^1 (X^2 ,\mathcal{K}_2 )$ with the cohomology of the complex \eqref{eqn:gersten} coming from Quillen's proof of Gersten's conjecture \cite{quillen}.
\begin{proposition}\label{prop:bigdiagram}
We have a commutative diagram of pointed sets with exact rows
\[
\begin{tikzcd}
\CH ^2 (X^2 ,1) \arrow[d] \arrow[r] & \mathcal{M}(X^2 ) \arrow[d] \arrow[r] & {\CH ^1 (X)_0 \times \CH ^1 (X)_0} \arrow[d] \arrow[r] & T(X^2 ) \arrow[d] \\
H^1 (K,V^{\otimes 2}/N) \arrow[r]           & {H^1 (K,U(\Q _p ,V,V^{\otimes 2}/N))_{s}} \arrow[r]           & H^1 (K,V)\times H^1 (K,V) \arrow[r]           & H^2 (K,V^{\otimes 2}).
\end{tikzcd}
\]
where the first and third vertical maps are the usual etale regulator maps and the second vertical map
\[
\mathcal{M}(X^2 ) \to H^1 (K,U(\Q _p ,V,V^{\otimes 2}/N))
\]
sends $(Z_1 ,Z_2 , \sum (C_i ,D_i ))$ (with $Z_1 \neq 0$ in $\CH ^1 (X_1 )$) to a subquotient of $H^2 _{\et }(X^2 - |C|\cup |Z_1 |\times X ,\Q _p (2))$.
\end{proposition}
\begin{proof}
Given $(Z_1 ,Z_2 , \sum  (C_i ,D_i ))$ as above, let $|C|$ and $|D|$ denote the supports of $\sum C_i $ and $D_i$ respectively. Then from the definition of cohomology with support relative to a closed subset we have a long exact sequence (see e.g. \cite[Proof of Theorem 6.3]{QC1})
\[
\ldots \to H^1 _{\et }(|C|; |D|,\Q _p )\to  H^2 _{\et }(X^2 ; |Z_1 | \times X \cup |C|,\Q _p ) \to H^2 _{\et }(X^2  ; |Z_1 |\times X,\Q _p ) \to \ldots .
\]
Using the equivalence relations above, without loss of generality, we may assume that the principal divisors $D_i$ are nonempty.
Then the principal divisors $D_i$ give a surjection
\begin{equation}\label{eqn:pushout}
H^1 _{\et }(|C|; |D|,\Q _p ) \to \Q _p .
\end{equation}
We have a K\"unneth decomposition
\[
H^2 _{\et }(X^2  ; |Z_1 |\times X,\Q _p )\simeq H^1 _{\et }(X;\Z _1 )\otimes H^1 _{\et }(X)\oplus H^2 (X).
\]
Taking the first K\"unneth component gives an exact sequence
\[
H^1 _{\et }(|C|;|D|,\Q _p )\to E_0 \to \Ker (H^1 _{\et }(X;\Z _1 )\otimes H^1 _{\et }(X)\to H^2 _{\et }(X))\to 0.
\]
Pushing out along \eqref{eqn:pushout} gives a short exact sequence
\[
0\to \Q _p \to E \to \Ker (H^1 _{\et }(X^2 ;|Z_1 |, \Q_p )\otimes H^1 _{\et } (X,\Q _p )\to H^2 _{\et }(X))\to 0.
\]
Hence, over a finite extensions over which all the irreducible components of $|Z_1 |$ are defined, by dualizing we have a mixed extension with graded pieces $\Q _p ,V^{\oplus n}$ and $H^1 _{\et }(X-|Z_1 |,\Q _p (1))\otimes H^1 _{\et }(X,\Q _p (1))$, where $n$ is the size of the support of $Z_1$.
For each $i$, the corresponding extension of $\Q _p $ by $V$ corresponds to the cycle class of $Z_2 $. Using the biextension structure on mixed extensions, we can form a mixed extension with graded pieces $\Q _p ,V, V^{\otimes 2}/N$.

The fact that the image of this mixed extension in $H^1 (K,V)\times H^1 (K,V)$ agrees with the Kummer map from $\CH ^1 (X)_0 ^2 $ follows from the description of the Kummer map given in the introduction.

In the case when we have something in the image of $\CH ^2 (X^2 ,1)$, this recovers the Scholl/Jannsen description of the \'etale regulator
\[
\CH ^2 (X^2 ,1)\to H^1 (K,H^2 _{\et }(X^2 ,\Z _p (2)))
\]
(see \cite[Appendix C.3]{jannsen1990}, \cite{scholl:ext}).
\end{proof}

\begin{remark}
Note that in general the mixed extensions $E(b,z)$ are \textit{not} in the subspace $H^1 (\Q ,U(\Q _p ,V,V^{\otimes 2} ))_s $. As we have have set things up, there are two ways this can fail: firstly if the class of the Gross--Schoen cycle in the Griffiths group is nontrivial \cite{GS}, and secondly because, by Lemma \ref{lemma:IE}, as $b$ and $z$ vary the class of $IE(b,z)$ varies by an element of the image of $H^1 (\Q ,V)^{\oplus 2}$ by $H^1 (\Q ,V^* \otimes V^{\otimes 2}/\Q _p (1))$. The second could easily be accommodated into the definition of mixed motivic extensions above, but the first requires a different set-up, which we outline below.
\end{remark}

\subsection{A generalisation}
More generally, if $X_1 $ and $X_2$ are smooth projective varieties of dimension $d_1 $ and $d_2 $ over a number field $K$, we can form the product
\begin{align*}
\CH ^{p_1 }(X_1 )\times \CH ^{p_2}(X_2 )\to \CH ^{p_1 +p_2 }(X_1 \times X_2 ) \\
(Z_1 ,Z_2 )\mapsto \pi _1 ^* Z_1 \cdot \pi _2 ^* Z_2 ,
\end{align*}
where $\pi _i$ is the projection $X_1 \times X_2 \to X_i $. Then Beilinson's conjecture implies that this pairing, restricted to $\CH ^{p_1 }(X_1 )_0 \times \CH ^{p_2 }(X_2 )_0 $, is torsion. If $\pi _1 ^* Z_1 \cdot \pi _2 ^* Z_2 $ is zero, then we get a union of codimension $p_1 +p_2 -1 $ cycles $W_i $ and rational functions $f_i$ on $W_i$ such that $\pi _1 ^* Z_1 \cdot \pi _2 ^* Z_2 =\sum \divv (f_i )$. Define $\mathcal{M}_{p_1 ,p_2 }(X_1 \times X_2 )$ to be the set of tuples $(Z_1 ,Z_2 ,\sum (W_i ,D_i ))$ where $Z_1 $ is in $Z^{p_1 }(X_1 )$, $Z_2 $ is in $Z^{p_2 }(X_2 )$, $W_i$ is a codimension $p_1 +p_2 -1$ subvariety of $X_1 \times X_2 $, $D_i$ is a principal divisor on $W_i $ such that $\sum _i D_i =\pi _1 ^* Z_1 \cdot \pi _2 ^* Z_2 $, modulo the equivalence relation generated by
\begin{align*}
(Z_1 ,Z_2 ,\sum (W_i ,D_i )) & \sim (Z_1 +\divv _{Y_1} (g_1),Z_2 ,\sum (W_i ,D_i )+(\pi _2 ^* Z_2 \cdot Y_1 ,\pi _2 ^* Z_2 \cdot \divv (g_1 ) )) \\
(Z_1 ,Z_2 ,\sum (W_i ,D_i )) & \sim (Z_1 ,Z_2 +\divv _{Y_2} (g_2),\sum (W_i ,D_i )+(\pi _1 ^* Z_1 \cdot Y_2 ,\pi _1 ^* Z_1 \cdot \divv (g_2 ))) \\
(Z_1 ,Z_2 ,\sum (W_i ,D_i )) & \sim (Z_1 ,Z_2,\sum (W_i ,D_i )+\partial (\{ h_1 ,h_2 \} )).
\end{align*} 
Here $Y_i$ is a codimension $p_i -1$ subvariety of $X_i$, $h_1 $ and $h_2$ are functions on a codimension $p_1 +p_2 -2$ subvariety $V$ of $X_1 \times X_2 $, and $\partial (\{ h_1 ,h_2 \})$ denotes the image of $\{ h_1 ,h_2 \} \in K^M _2 (K(V))$ in $\oplus _{x\in X^{p_1 +p_2 -1}}\Div (\underline{x})$ under the residue map.
We get an exact sequence of pointed sets
\[
\CH ^{p_1 +p_2 }(X_1 \times X_2 ,1)\to \mathcal{M}_{p_1 ,p_2 }(X_1 \times X_2 )\to \CH ^{p_1 }(X_1 )_0 \times \CH ^{p_2 }(X_2 )_0 \to F^2 \CH ^{p_1 +p_2 }(X_1 \times X_2 )
\]
(for our purposes, we can take the definition of $F^2 \CH ^{p_1 +p_2 }(X_1 \times X_2 )$ to be any subgroup of $\CH ^{p_1 +p_2 }(X_1 \times X_2 )_0 $ mapping to zero in $H^1 (K ,H^{2(p_1 +p_2 )-1}(X_1 \times X_2 ,\Q _p (p_1 +p_2 )))$ containing the image of $\CH ^{p_1 }(X_1 )_0 \times \CH ^{p_2 }(X_2 )_0 $).
We also have an exact sequence of pointed sets
\[
H^1 (K,H_1 \otimes H_2 )\to H^1 (K,U(\Q _p ,H_1 ,H_1 \otimes H_2 )_s )\to H^1 (K,H_1 )\times H^1 (K,H_2 )\to H^2 (K, H_1 \otimes H_2 )
\]
where $H_i :=H^{2d_i -1}(X_i ,\Q _p (d_i ))$, and $U(\Q _p ,H_1 ,H_1 \otimes H_2 )_s$ is the subgroup of $U(\Q _p ,H_1 ,H_1 \otimes H_2 )$ consisting of automorphisms whose projection to $\Hom (H_1 ,H_1 \otimes H_2 )$ lies in the image of $H_2 $. The relation between the two exact sequences appears to be more complicated than in Proposition \ref{prop:bigdiagram}, because the coniveau filtration on $H^{2(d_1 +d_2 -1)}(X_1 \times X_2 )$ is more complicated than that of $H^2$ of a surface.

\section{An example}\label{sec:example}
In this section we give an example where `exhibiting the torsion-ness' of an element of the Albanese kernel enables one to find divisors on curves on $X\times X$ which determine the generalised height pairing on $X\times X$, in a situation where it is not clear how to determine the pairing using rational (or even algebraic) points on the curve $X$. The example is somewhat ad hoc. For more systematic methods for exhibiting torsion, see \cite{GL23} and the references therein.
%
%

For example, consider the curve
\[
X:y^2 =f(x):=4x^5 + 8x^4 + 16x^3 + 12x^2 + 8x + 1.
\]
This has Mordell--Weil rank $2$, and $5$ obvious rational points
\[
\{ \infty ,(0,\pm 1),(1,\pm 7)\}.
\]
Hence rational points allow us to compute $h(P_1 ^2 )$ and $h(P_2 ^2 )$, where $P_1 =(0,1)-\infty $ and $P_2 =(1,7)-\infty $. Recall that this is because $h(P_i ^2 )$ is the generalised height of the mixed extension formed from $E_2 (X;\infty ,P_i )$. How can we compute the generalised height $h(P_1 P_2 )$?

Consider the (normalisation of) the curve
\[
C:y^2 =f(x-1/2),z^2 =f(-1/2-x).
\]
inside $X\times X$ (i.e. the normalisation of the subvariety given by the closure of 
\[
\{ ((x_1 ,y_1 ),(x_2 ,y_2 ))\in (X-\infty )^2 :x_1 +x_2 =-1 \}
\]
in $X\times X$. The Jacobian of $C$ is isogenous to $J\times J\times J_1 \times J_2 $, where $J_1 $ is the Jacobian of the curve 
\[
C_1 :y^2 =-16x^5 - 76x^4 - 122x^3 - \frac{147}{2}x^2 - \frac{181}{16}x + \frac{169}{64},
\]
$J_2$ is the Jacobian of the curve
\[
C_2 :y^2 =-16 - 76x - 122x^2 - \frac{147}{2}x^3 - \frac{181}{16}x^4 + \frac{169}{64}x^5 ,
\]
and as usual $J$ is the Jacobian of $X$.
The isogeny decomposition is realised by the maps $C\to C_i $ sending $(x,y,z)$ to $(x^2 ,yz)$ and $(1/x^2 ,yz/x^5 )$ respectively.

The Mordell--Weil rank of the Jacobian of $C_1$ is $2$, and it contains the rational points 
\[
\{ \infty ,(0,\pm -13),(-3/4,\pm 1),(-7/4,\pm 1),(-31/4,\pm 475) \}.
\]
We have the relation
\[
5\cdot (-3/4,1)-(-7/4,1)-(-31/4,475)-3\cdot \infty =0
\]
in $\Jac (C_1 )$. Pulling back to $C$ we obtain the principal divisor
\begin{align*}
D_0 = & 5(\frac{\sqrt{-3}}{2},1,1)+5(\frac{\sqrt{-3}}{2},-1,-1)+5(-\frac{\sqrt{-3}}{2},1,1)+5(-\frac{\sqrt{-3}}{2},-1,-1)\\
& -(\frac{\sqrt{-7}}{2},1,1)-(\frac{\sqrt{-7}}{2},-1,-1)-(-\frac{\sqrt{-7}}{2},1,1)-(-\frac{\sqrt{-7}}{2},-1,-1) \\
& -(\frac{\sqrt{-31}}{2},-14-3\sqrt{-31},-14+3\sqrt{-31}) \\
& -(\frac{\sqrt{-31}}{2},14+3\sqrt{-31},14-3\sqrt{-31}) \\
& -(-\frac{\sqrt{-31}}{2},14+3\sqrt{-31},14-3\sqrt{-31})-(-\frac{\sqrt{-31}}{2},-14-3\sqrt{-31},-14+3\sqrt{-31}) \\
& - 6\infty ^+ -6\infty ^- .
\end{align*}
Here $\infty ^+$ and $\infty ^-$ are the two points of $C$ mapping to $(\infty ,\infty )$ in $X\times X$. It follows that on the hyperelliptic curve
\[
C_3 :y^2 =-16x^10 - 76x^8 - 122x^6 - \frac{147}{2}x^4 - \frac{181}{16}x^2 + \frac{169}{64},
\]
We claim that the image of $D_0$ in $\Sym ^2 J(\Q )$, together with the classes $P_1 ^2$ and $P_2 ^2 $, span $\Sym ^2 J(\Q )\otimes \Q $, and hence that the generalised height of $D_0 $, together with those of $P_1 $ and $P_2$, allow us to determine the generalised height pairing on $\Jac (X)\times \Jac (X)$.
To see this, note that
\begin{align*}
&\Tr ((\zeta _3 ,1)-\infty )\sim -\frac{1}{4}P_2-\frac{1}{2}P_1 \\
&\Tr ((\frac{-1+\sqrt{-7}}{2},1)-\infty )\sim -\frac{1}{4}P_2+\frac{1}{2}P_1 \\
&\Tr ((\frac{-1+\sqrt{-31}}{2},14+3\sqrt{-31})-\infty )\sim  3P_1 . \\
\end{align*}
Hence the class of $D_0$ is given by 
\begin{align*}
& \frac{5}{16}(P_1 ^2 +4P_2 ^2 +4P_1 P_2 )-\frac{1}{16}(P_1 ^2 +4P_2 ^2 -4P_1 P_2 )-9P_2 ^2 \\
& = \frac{1}{4}P_1^2 + \frac{3}{2} P_1 P_2 - 8P_2 ^2 .
\end{align*}

\section{Other remarks}\label{sec:rubbish}
\subsection{Relative completions and motivic Galois groups}
The description given above of the generalised height has the slightly awkward feature of involving mixed extensions of $\Q _p ,V^{\oplus n}$ and $V^{\otimes 2}$ for all $n$. In fact there is a more natural way to describe the Galois cohomology of $U(\Q _p ,V^{\oplus n},V^{\otimes 2})$ (and mixed extensions more generally) which clarifies the situation. Namely, one can use Hain's relative completion \cite{hain:non} to describe Selmer schemes as \textit{algebraic} nonabelian cohomology of algebraic groups arising as suitable (weighted) relative completions $\mathcal{G}$ of Galois groups of $\Q $ with restricted ramification. It follows that fundamental objects are the isotypic components of the unipotent radical of $\mathcal{G}$. This approach is closely related to the recent preprint of Corwin \cite{corwin}.

Hain's work also has the appealing feature of being formulated in terms of sections of an exact sequence, and hence being applicable without the assumption of a rational point. It would be natural to have an algebraic cycle version, which extends (and improves) the cycle-theoretic construction in this paper and includes the cycle-theoretic obstructions of Esnault and Wittenberg \cite{EW}.
\subsection{An analogue for the unit equation}\label{subsec:BS}
In \cite{DCW}, Dan-Cohen and Wewers explain how to determine equations for $X(\Z _p )_{S,2}$ when $X=\mathbb{P}^1 -\{ 0,1,\infty \}$ and $S$ is a set of primes of $\Q $ using standard bilinearity properties of the $p$-adic dilogarithm function. The analogue, for the dilogarithm, of exhibiting that an element in the Albanese kernel of a product of curves is torsion is to exhibit that the elements $\{ p_1 ,p_2 \}$ (for $p_1 ,p_2 \in S$) are torsion in $K_M ^2 (\Q )$. Concretely, this means finding a positive integer $N$, and $n_i \in \Z ,z_i \in \Q ^\times -\{ 1\}$ such that
\[
N\cdot [p_1 ] \otimes [p_2 ]= \sum n_i [z_i ]\otimes [1-z_i ]
\]
in $\Q ^\times \otimes _{\Z }\Q ^\times $. An algorithm for this is given in \cite{milnor}.

One can rephrase this in a language of `motivic higher Albanese sets' as follows. Recall \cite{suslin} the Bloch group of a field $K$ is defined to be the cohomology of the complex
\begin{align*}
& \langle [x]+[y]+[(1-x)/(1-xy)]+[1-xy]+[(1-y)/(1-xy)] |x,y \in K^\times -\{ 1\}, x\neq y^{-1} \rangle \\
& \to \Z [K^\times -\{ 1\} ] \stackrel{[x]\mapsto x\wedge (1-x)}{\longrightarrow}  \wedge ^2 K^\times 
\end{align*}
 Define $\widetilde{B}$ to be the set of triples $(x,y,\sum n_i [z_i ])$ in $K^\times \times K^\times \times \Z [K^\times -1]$ such that
\[
x\otimes y =\sum n_i z_i \otimes (1-z_i )
\]
modulo the equivalence relation
\[
(x,y,\sum n_i [z_i ])\sim (x',y',\sum n' _i [z' _i ])
\]
if $\sum n_i [z_i ]-\sum n_i ' [z_i ' ]$ lies in the subspace of $\Z [K^\times -1]$ generated by 
\[
[a]+[b]+[\frac{1-a}{1-ab}]+[1-ab]+[\frac{1-b}{1-ab}].
\]
We get a commutative diagram of pointed sets with exact rows
\begin{small}
\[
\begin{tikzcd}
1 \arrow[r] & B_2 (K) \arrow[d] \arrow[r] & \widetilde{B} \arrow[d] \arrow[r] & K^\times \times K^\times \arrow[d] \arrow[r] & K^M _2 (K) \arrow[d]  \\
1 \arrow[r] & H^1  (K,\Z _p (2)) \arrow[r]           & H^1  (K,U(\Z _p ,\Z _p (1),\Z _p (2))) \arrow[r]           & H^1  (K,\Z _p (1))^2 \arrow[r]           &     H^2 (K,\Z _p (2))        
\end{tikzcd}
\]
\end{small}
\subsection{Geometric analogues}\label{subsec:geom}
Let $S$ be a smooth projective curve with generic point $\eta $ and $\pi :X\to S$ a smooth relative curve. Then the height pairing on $X_{\eta }$ coincides with the intersection product
\[
\CH ^1 (X)\times \CH ^1 (X)\to \CH ^2 (X)
\]
composed with the push forward $\CH ^2 (X)\to \CH ^1 (S)$. An obvious geometric analogue of the product considered in section 3 is the map
\[
\CH^1 (X)\times \CH ^1 (X)\to \CH^2 (X\times _S X)
\]
defined by sending $(Z_1 ,Z_2 )$ to $(\pi _1 ^* Z )\cdot (\pi _2 ^* Z)$, where $\pi _1 $ and $\pi _2 $ are the projections $X\times _S X\to X$. Similarly, one could consider the product
\[
\CH ^1 (X)\times \CH ^2 (X\times _S X\times _S X)\to \CH ^2 (X\times _S X)
\]
as a geometric version of the generalised heights arising in quadratic Chabauty.
\subsection{Circle-valued generalised heights}
Given a notion of `motivic' mixed extensions, one can replace $p$-adic generalised heights with Archimedean generalised heights. This leads to `pairings'
\begin{align*}
& \Ker (\CH ^{d_1 }(X_1 )_0 \times \CH^{d_2 }(X_2 )_0 \to \CH ^{d_1 +d_2 }(X_1 \times X_2 )) \\
&  \to \Ext ^1 _{\mathbb{R}-\MHS }(\R ,(H^{2d_1 -1}(X_1 ,\R (d_1 ))\otimes H^{2d_2 -1}(X_2 ,\R (d_2 )))/N) /\reg (\CH ^{d_1 +d_2 }(X_1 \times X_2 ,1)),
\end{align*}
$X_i$ are smooth projective varieties of dimension $d_i$, $N\subset H^{2d-1}(X,\R (d_1 ))\otimes H^{2d_2-1}(X,\R (d_2 ))$ is the image of $H^{2(d_1 +d_2 -1)}$ of all codimension $d_1 +d_2 -1$ subvarieties,
which may be defined as follows. Taking cohomological realisations we obtain a mixed extension of $\mathbb{R}$-mixed Hodge structures with graded pieces $\R ,H^{2d-1}(X_1 ,\R (d_1 ))$ and $(H^{2d_1 -1}(X_1 ,\R (d_1 ))\otimes H^{2d_2 -1}(X_2 ,\R (d_2 )))/N$. Mixed extensions of $\R $-mixed Hodge structures were recently studied by Hain \cite{hain2024periods}. Generalising the usual construction for the height pairing (as in \cite{scholl}) he shows how one may associate to such a mixed extension an element of $\ext ^1 _{\R -\MHS }(\R ,(H^{2d_1 -1}(X_1 ,\R (d_1 ))\otimes H^{2d_2 -1}(X_2 ,\R (d_2 )))/N)$, called the real period of the mixed extension.

In general, if the analogy in \ref{subsec:geom} is reasonable, it is natural to ask whether the product constructed is related to the intersection product on arithmetic Chow groups \cite{BG}
\[
\widehat{\CH }^{d_1 }(X_1 )\times \widehat{\CH } ^{d_2 }(X_2 )\to \widehat{\CH }^{d_1 +d_2 }(X_1 \times X_2 ).
\]
Even more speculatively, is there an interpretation of the usual generalised height as a pairing on (a suitable notion of) $p$-adic arithmetic Chow groups?

Assuming the Beilinson conjectures \cite{beilinson:height} \footnote{and imposing appropriate integrality conditions on our cycles.} we get a pairing somewhat reminiscent of the Mazur--Tate circle pairing \cite{MT}. It is natural to wonder (paraphrasing \cite{BDMT}) what meaning (if any) can be ascribed to this: for example, is there a `circle-pairing Beilinson conjecture' generalising the work of Bertolini and Darmon on the Mazur--Tate circle pairing \cite{BDMT}?

\subsection{$p$-descent for Bloch--Kato Selmer groups}
In \cite{dogra2023}, \cite{dogra20242} and \cite{berry2025refined}, the problem of bounding the dimension of $H^1 _f (\Q ,\wedge ^2 V)$ via the construction of a suitable `Selmer' subspace of $H^1 (\Q ,\wedge ^2 J[2])$ is considered. This is somewhat analogous to bounding the rank of an abelian variety via $2$-descent. The analogue of the usual Kummer exact sequence for an abelian variety gives a commutative diagram
\[
\begin{small}
\begin{tikzcd}
0 \arrow[r]  & \CH ^2 (X^2 ,1)\otimes \F _2 \arrow[d] \arrow[r] & \CH ^2 (X^2 ,1;2) \arrow[d] \arrow[r] & \CH ^2 (X^2 )[2] \arrow[d] \arrow[r] & 0 \\
0 \arrow[r] & H^1 (K,H)\otimes \F _2 \arrow[r] &  H^1 (K,H\otimes \F _2 )   \arrow[r]  & H^2 (K,H)[2] \arrow[r]     & 0
\end{tikzcd}
\end{small} \]
\noindent where $H:=H^2 _{\et }(X^2 ,\Z _2 (2))$. Hence exhibiting nontrivial $2$-torsion in the Chow group of $X^2 $ has a role in (full-fat) quadratic Chabauty analogous to that of exhibiting nontrivial $2$-torsion in the Tate--Shafarevich group of the Jacobian in classical Chabauty.
\subsection{Other varieties}
In the above we used the curves on the surface $X\times X$ to construct mixed extensions. It is natural to wonder whether there is a `simpler' surface (or simpler variety) to work with. In the case of a genus 2 curve $X$, two natural candidates are its Jacobian $J$, and its Kummer variety $K$, since their $H^2$ captures the `interesting part' (c.f. Lemma \ref{lemma:sym_boring}) of $H^2 (X\times X)$. 

This gives a simpler presentation of the example in section \ref{sec:example} as follows. Instead of working with the genus $8$ curve on $X\times X$, one can instead work with its image in $\Jac (X)$, which is a curve of genus 4. The main computational problem is then to find principal divisors on curves which exhibit the torsion-ness of zero-cycles on $\Jac (X)$. For our genus 4 curve, using the hyperelliptic involution this reduces to finding principal divisors on its image in the Kummer variety which lift to rational points on $\Jac (X)$. Hence recent work of Gazaki and Love \cite{GL23} suggests a systematic way to compute generalised heights, by finding \textit{rational curves} on the Kummer variety whose pullbacks contain rational points on $\Jac (X)$.
\bibliography{bib_Kim}
\bibliographystyle{alpha}

\end{document}